\documentclass[11pt]{article}
\usepackage[centertags]{amsmath}
\usepackage{amsfonts}
\usepackage{amssymb}
\usepackage{amsthm,color}
\usepackage{newlfont}
\usepackage{bbm}

\newtheorem{Theorem}{Theorem}[section]
\newtheorem{Definition}[Theorem]{Definition}

\newtheorem{Lemma}[Theorem]{Lemma}

\newtheorem{Remark}[Theorem]{Remark}

\newtheorem{Hypothesis}{Hypothesis}

\numberwithin{equation}{section}

\DeclareMathOperator{\diver}{div}

\begin{document}

\def\le{\left}
\def\r{\right}
\def\cost{\mbox{const}}
\def\a{\alpha}
\def\d{\delta}
\def\ph{\varphi}
\def\e{\epsilon}
\def\la{\lambda}
\def\si{\sigma}
\def\La{\Lambda}
\def\B{{\cal B}}
\def\A{{\mathcal A}}
\def\L{{\mathcal L}}
\def\O{{\mathcal O}}
\def\bO{\overline{{\mathcal O}}}
\def\F{{\mathcal F}}
\def\K{{\mathcal K}}
\def\H{{\mathcal H}}
\def\D{{\mathcal D}}
\def\C{{\mathcal C}}
\def\M{{\mathcal M}}
\def\N{{\mathcal N}}
\def\G{{\mathcal G}}
\def\T{{\mathcal T}}
\def\R{{\mathbb R}}
\def\I{{\mathcal I}}

\def\bw{\overline{W}}
\def\phin{\|\varphi\|_{0}}
\def\s0t{\sup_{t \in [0,T]}}
\def\lt{\lim_{t\rightarrow 0}}
\def\iot{\int_{0}^{t}}
\def\ioi{\int_0^{+\infty}}
\def\ds{\displaystyle}
\def\pag{\vfill\eject}
\def\fine{\par\vfill\supereject\end}
\def\acapo{\hfill\break}

\def\beq{\begin{equation}}
\def\eeq{\end{equation}}
\def\barr{\begin{array}}
\def\earr{\end{array}}
\def\vs{\vspace{.1mm}   \\}
\def\rd{\reals\,^{d}}
\def\rn{\reals\,^{n}}
\def\rr{\reals\,^{r}}
\def\bD{\overline{{\mathcal D}}}
\newcommand{\dimo}{\hfill \break {\bf Proof - }}
\newcommand{\nat}{\mathbb N}
\newcommand{\E}{\mathbb E}
\newcommand{\Pro}{\mathbb P}
\newcommand{\com}{{\scriptstyle \circ}}
\newcommand{\reals}{\mathbb R}

\newcommand{\red}[1]{\textcolor{red}{#1}}

\def\Amu{{A_\mu}}
\def\Qmu{{Q_\mu}}
\def\Smu{{S_\mu}}
\def\H{{\mathcal{H}}}
\def\Im{{\textnormal{Im }}}
\def\Tr{{\textnormal{Tr}}}
\def\E{{\mathbb{E}}}
\def\P{{\mathbb{P}}}
\def\span{{\textnormal{span}}}
\title{A generalization of the Freidlin-Wentcell theorem on averaging of Hamiltonian systems}
\author{Yichun Zhu\\
\normalsize University of Maryland, College Park\\ United States}
\date{}

\date{}

\maketitle

\begin{abstract}
In this paper, we generalize the classical Freidlin-Wentzell's theorem for random perturbations of Hamiltonian systems. In stead of the two-dimensional standard Brownian motion, the coefficient for the noise term is no longer the identity matrix but a state-dependent matrix plus a state-dependent matrix that converges uniformly to 0 on any compact sets as $\e$ tends to 0. We also take the drift term into consideration where the drfit term also contains two parts, the state-dependent mapping and a state-dependent mapping that converges uniformly to 0 on any compact sets as $\e$ tends to 0. In the proof, we use the result of generalized differential operator. We also adapt a new way to prove the weak convergence inside the edge by constructing an auxiliary process and apply Girsanov's theorem in the proof of gluing condition.
\end{abstract}

\section{Introduction}
Consider the following system 
\begin{equation}\label{qsde}
\begin{cases}
d q^\e_t = \frac{1}{\e} g (q^\e_t) dt + [b(q^\e_t)+b^{\e}(q^\e_t)] dt + [\sigma(q^\e_t)  + \sigma^\e(q^\e_t)] dW_t ,\\
q^\e_0 = q \in\,\mathbb{R}^2,
\end{cases}
\end{equation}
where $\e$ is a small positive constant and $W_t$ is a standard two-dimensional Brownian motion defined on the stochastic basis $(\Omega, \mathcal{F}, \{\mathcal{F}_t\}_{t \geq 0}, \mathbb{P})$. We assume that $b, b^\e:\mathbb{R}^2\to \mathbb{R}^2$ and $\si, \si^\e:\mathbb{R}^2\to \mathbb{R}^{2\times 2}$ are differentiable mappings having bounded derivatives and the mappings $b^\e$ and $\si^\e$ converge to zero, as $\e$ goes to zero, uniformly on any compact subset of $\mathbb{R}^2$.

In what follows, we shall assume that there exists a function $H: \mathbb{R}^2 \mapsto \mathbb{R}^2$ such that
\begin{equation}\label{Hamiltonian}
g(x) \cdot \nabla H(x)= 0,\ \ \ \  x \in \mathbb{R}^2.
\end{equation}
This means in particular that if we denote by $X^\e_t$ the solution of the unperturbated system 
\begin{equation*}
\frac{dX_t^\e}{dt} = \frac{1}{\e} g(X_t^\e),
\end{equation*}
then 
$H(X^\e_t)=H(X^\e_0)$, for every $t\geq 0$ and $\e>0$. This means that $X^\e_t$ remains on the same level set of $H$, for every $t\geq 0$.
Moreover, if we define
\begin{equation}\label{invariant measure}
a(x) = |g(x)|/|\nabla H(x)|,
\end{equation}
then it can be easily proved that $a^{-1}(x)$ is the density of the invariant measure for  $X^\e_t$ .

Now, for every $x \geq \inf_{z\in \mathbb{R}^2}H(z)$, we denote by $C(x)$ the $x$-level set of $H$, that is
\[C(x)=\le\{z \in\,\mathbb{R}^2\,:\,H(z)=x\r\}.\]
The set  $C(x)$ may consist of several connected components
\[C(x)=\bigcup_{k=1}^{N(x)} C_k(x),\]
and for every $z \in\,\mathbb{R}^2$  we will denote by $C_{k(z)}(H(z))$  the connected component of the level set $C(H(z))$, to which the point $z$ belongs. If we identify all points in $\mathbb{R}^2$  belonging to the same connected component of a given level set $C(x)$ of the Hamiltonian $H$, we obtain a graph $\Gamma$,  given by several intervals $I_1,\ldots I_n$ and vertices $O_1,\ldots, O_m$. In what follows, we shall denote by $\Pi:\mathbb{R}^2\to \Gamma$ the {\em identification map}, that associates to every point $z \in\,\mathbb{R}^2$ the corresponding point $\Pi(z)$ on the graph $\Gamma$. We have $\Pi(z)=(H(z),i(z))$, where $i(z)$ denotes the number of the interval on the graph $\Gamma$, containing the point $\Pi(z)$.

In the present paper we are interested in the asymptotic behavior of the $\Gamma$-valued process $\Pi(q^\e)$. Namely, we want to show that $\Pi(q^\e)$ converges in distribution in the space $C([0,T];\Gamma)$,  as $\e\to 0$, to a Markov process in $\Gamma$, whose generator is explicitly described in terms of  suitable differential operators in the interior of every edge and suitable gluing conditions at each interior vertex.  

If we define $x^\e_t=H(q^\e_t)$, as an immediate consequence of the It\^o's formula we have
\begin{equation*}
dx^\e_t=\mathcal{L}_0 H (q^\e_t) dt + \mathcal{R}_0 H (q^\e_t) dW_t + \mathcal{L}_0^\e H (q^\e_t)dt + \mathcal{R}_0^\e H (q^\e_t)dW_t
\end{equation*}
where
\begin{equation*}
\mathcal{L}_0 f (x)=\nabla f(x)^t \cdot b(x) + \frac{1}{2} \sum_{i,j} (\sigma \sigma^*)_{i,j}(x) \frac{\partial^2 f(x)}{\partial x_i \partial x_j},
\end{equation*}
\begin{equation*}
\mathcal{R}_0 f(x)=\nabla f(x)^t \sigma(x),
\end{equation*}
\begin{equation*}
\mathcal{L}_0^\e f(x)=\nabla f(x)^t \cdot b^\e(x) + \frac{1}{2} \sum_{i,j} [\sigma (\sigma^\e)^* +  \sigma^\e \sigma^* + \sigma^\e (\sigma^\e)^*]_{i,j}(x) \frac{\partial^2 f(x)}{\partial x_i \partial x_j},
\end{equation*}
and
\begin{equation*}
\mathcal{R}_0^\e f (x)=\nabla f(x)^t \sigma^\e(x).
\end{equation*}

Next, for every $x \in I_i$, we define
\begin{equation*}
A_i(x)=\frac{1}{T_i(x)}\oint_{C_i(x)}{\mathcal{R}_0 H (u) \mathcal{R}_0 H(u)^*\frac{dl}{|g(u)|}}
\end{equation*}
and
\begin{equation*}
B_i(x)=\frac{1}{T_i(x)}\oint_{C_i(x)} {\mathcal{L}_0 H (u)  \frac{dl}{|g(u)|}},
\end{equation*}
where
\begin{equation*}
T_i(x)={\oint_{C_i(x)} \frac{dl}{|g(u)|}}.
\end{equation*}
Moreover, we define
\begin{equation}\label{d1}
\mathcal{L}_i f(x) = B_i(x)\frac{d}{dx}f + \frac{1}{2} A_i(x) \frac{d^2}{dx^2}f.
\end{equation}
With these notations, we can introduce the following operator acting on functions defined on the graph $\Gamma$.

\begin{Definition}\label{def gluing condition}
For $I_i \sim O_k$, let
\begin{equation}\label{gluing probability}
p_{ki}=\frac{\beta_{ki}}{\sum_{i: I_i \in O_k} \beta_{ki}},
\end{equation}
where 
\begin{equation*}
\beta_{ki}=\oint_{C_{ki}} \frac{|\nabla H(x) \sigma\sigma^*(x) \nabla H(x)|}{|g(x)|}dl.
\end{equation*}
We denote by $D(L)\subseteq C(\Gamma; \mathbb{R}) $ the set consisting all functions  $f$ defined on the graph $\Gamma$ such that $\mathcal{L}_jf$ is well defined in the interior of the edge $I_j$ and for every $I_j\sim O_k$ there exists finite
\[\lim_{x\to O_k}\mathcal{L}_j f(x)\]
and the limit is independent of the edge $I_j$. Moreover, for each interior vertex $O_k$ 
\[\sum_{j\,:\,I_j\sim O_k} \pm \rho_{kj}f_j^\prime(H(O_k))= 0,\]
where $f_j^\prime$ denotes the derivative of $f$ with respect to the local coordinate $\la$, along the edge $I_j$ and the signs $\pm$ are taken if $H>H(O_k)$ or $H<H(O_k)$.

Next, for every $f \in\,D(L) $, we  define
\[L f(x)=\begin{cases}
\mathcal{L}_j f(x),  &  \text{if}\ x\ \text{is an interior point of}\ I_j,\\
\lim_{x\to O_k}\mathcal{L}_j f(x),   &  \text{if}\ x\ \text{is the vertex}\ O_k\ \text{and}\ I_j\sim O_k.
\end{cases}\]
\end{Definition} 
The main result of this paper is given by the following theorem.
\begin{Theorem}\label{main}
Suppose $q^\e_t$ satisfies the following stochastic differential equation
\begin{equation*}
d q^\e_t = \frac{1}{\e} g (q^\e_t) dt + b(q^\e_t) dt +\sigma(q^\e_t)  dW_t + b^{\e}(q^\e_t) dt + \sigma^\e(q^\e_t) dW_t \\
\end{equation*}
with initial condition $q^\e_0=q$. Assume the coefficients satisfy Hypothesis \ref{h1} and the Hamiltonian $H: \mathbb{R}^2 \to \mathbb{R}$ introduced in   (\ref{Hamiltonian}) satisfies Hypothesis \ref{h2}. Then the process $\Pi(q^\e_t)=(x^\e_t, i(q^\e_t))$ converges weakly  in $C([0,T];\Gamma)$ to the Markov process $Y$ generated by the operator $(L,D(L))$, as described in Definition \ref{def gluing condition}.
\end{Theorem}

In the present paper, we generalize the well known result by Freidlin-Wentcell on the validity of an averaging principle for Hamiltonian sustems (see \cite[Chapter 8]{fw} )to a more general case and introduce a new method which simplifies some steps in the proof. Compared with the original Freidlin-Wentcell theorem, here we can cover the case of  a state dependent diffusion coefficient and we can also deal with a the drift term. Moreover, both the diffusion coefficient and the drift term are given by the sum of a term of order one and a term of order $\e$. For the terms of order $\e$, we assume that, as $\e$ goes to zero, they converge uniformly to zero over any compact sets in $\mathbb{R}^2$.  

In the proof of the weak convergence in the interior of every edge and the analysis of the behavior of the process near the exterior vertices, we introduce  a new proof, based on the construction of a suitable auxiliary process. What is remarkable is that, unlike the original proof, this new method  unifies the two cases together. In the proof of the svalidity of the gluing conditions, when dealing with the extra terms of order $\e$ in the drift and in the diffusion coefficient, we first introduce  an auxiliary vector field in order to be able to apply the classical results based on  generalized differential operators. Then we apply Girsanov's theorem to get rid of the compensated drift term.

Finally, we would like to mention that our main motivation in studying this type of problem is provided by the paper \cite{cwz}, where  together with Cerrai and Wehr we prove the validity of the Smoluchowskii-Kramer approximation for a system with a finite number of degrees of freedom, in the presence of a state dependent magnetic field $\lambda$. In this case, the problem is regularized by adding a small friction of intensity $\e>0$. After the small mass limit for the regularized problem is taken, we obtain a Hamiltonian system (with Hamiltonian $\la$)  perturbed by a deterministic and a stochastic term. Both perturbations are given by the sum of two terms of different order, that with the notations of the present paper correspond to the drifts $b$ and $b^\e$ and the diffusions $\sigma$ and $\sigma^\e$. Theorem \ref{main} allows us to obtain the limiting behavior, as $\e$ goes to zero, for the slow component of the solution of the Hamiltonian system obtained from the small mass limit of the regularized problem. As shown in Theorem \ref{main}, the limiting process is given by a suitable Markov process on the graph associated with the Hamiltonian $\lambda$.

\section{Some Preliminaries}
\subsection{Hypotheses and notations}

Concerning the coefficients in equation \eqref{qsde}, we assume the following conditions
\begin{Hypothesis} \label{h1}
\begin{enumerate}
\item The mappings $g,\ b,\ b_\e: \mathbb{R}^2 \to \mathbb{R}^2$ and $\sigma ,\ \sigma_\e:  \mathbb{R}^2 \to \mathbb{R}^{2 \times 2}$ are all continuously differentiable  with uniformly bounded derivatives.
\item The functions $b^\e$ and $\sigma^\e$ converge to zero as $\e$ goes to zero,  uniformly on any compact set in $\mathbb{R}^2$.
\end{enumerate}
\end{Hypothesis}

The Hamiltonian $H$ satisfies the following conditions.
\begin{Hypothesis} \label{h2}
$H$ belongs to $\textbf{C}^{4} (\mathbb{R}^2)$ and has bounded second derivatives. Moreover
\begin{enumerate}
\item $H$ has finite many critical points and for any two critical points $x_1$ and $x_2$, $H(x_1) \neq H(x_2)$.
\item For any critical point of $H$, the matrix $\sigma$ is invertible in some neighbor of it.
\item The matrix of second order derivative is non-degenerate at any critical point of H.
\item There exists a constant $C_1>0$ such that $H(x)\geq C_1(|x|^2+1)$, $\nabla H(x) \geq C_1|x|$, and $\Delta H(x) \geq C_1$, for all $x\in \mathbb{R}^2$ such that $|x|$ large enough.
\end{enumerate}
\end{Hypothesis}

As we mentioned in the Introduction, if we identify all points in $\mathbb{R}^2$  belonging to the same connected component of a given level set $C(x)$ of the Hamiltonian $H$, we obtain a graph $\Gamma$,  given by several intervals $I_1,\ldots I_n$ and vertices $O_1,\ldots, O_m$. The vertices will be of two different types,  external and internal vertices. External vertices correspond to local extrema of  $H$, while internal vertices correspond to saddle points of $H$. Among external vertices, we will also include $O_\infty$, the endpoint of the interval in the graph corresponding to the point at infinity.

We have seen that the {\em identification map} associates to every point $z \in\,\mathbb{R}^2$ the corresponding point $\Pi(z)$ on the graph $\Gamma$. Thus, if $\Pi(z)=(H(z),i(z))$, then $i(z)$ denotes the number of the edge on the graph $\Gamma$, containing the point $\Pi(z)$. If $O_i$ is one of the interior vertices, the second coordinate cannot  be chosen in a unique way, as there are three edges having $O_i$ as their endpoint. 

On the graph $\Gamma$, a distance can be introduced in the following way. If $y_1=(x_1,i)$ and $y_2=(x_2,i)$ belong to the same edge $I_i$, then $d(y_1,y_2)=|x_1-x_2|$. In the case $y_1$ and $y_2$ belong to different edges, then
\[d(y_1,y_2)=\min\,\le\{d(y_1, O_{i_1})+d(O_{i_1},O_{i_2})+\cdots+d(O_{i_j},y_2)\r\},\]
where the minimum is taken over all possible paths from $y_1$ to $y_2$, through every possible sequence of vertices $O_{i_1},\ldots,O_{i_{j}}$, connecting $y_1$ to $y_2$.

If $x$ is not a critical value of $H$, then each $C_k(x)$ consists of one periodic trajectory of the vector field $g(z)$. If $x$ is a local extremum of $H(z)$, then, among the components of $C(x)$ there is a set consisting of one point, the rest point of the flow. If $H(x)$ has a saddle point at some point $z_0$ and $H(z_0)=x$, then $C(x)$ consists of three trajectories, the equilibrium point $z_0$ and the two trajectories that have $z_0$ as their limiting point, as $t\to \pm \infty$.

We introduce some other notations that will be used through out the paper. Letter $D$ is used to denote domain, while letter $C$ is for the level set of the Hamiltonian system. For any $A \subseteq \mathbb{R}$, $D(A)$ is the interior of $\{ z \in \mathbb{R}^2 ; H(z) \in A \}$
\begin{equation*}
D_i= \{ z \in \mathbb{R}^2; \Pi(z)= (\cdot, i) \}
\end{equation*}
\begin{equation*}
D_i(H_1, H_2)= D_i \cap D((H_1, H_2))
\end{equation*}
\begin{equation*}
D_k( \pm \d)= D((H(O_k)-\d, H(O_k)+\d)),
\end{equation*}
and
\begin{equation*}
C(x)= \partial D((-\infty, x])
\end{equation*}
\begin{equation*}
C_i(x)=\bar{D}_i \cap C(x)
\end{equation*}
\begin{equation*}
C_k=C(H(O_k))
\end{equation*}
\begin{equation*}
C_{ki}=C(H(O_k)) \cap \partial D_i
\end{equation*}
\begin{equation*}
C_{ki}(\pm \d)=\partial D_k(\pm \d) \cap I_i.
\end{equation*}
Finally, we write $I_i \sim O_k$ if and only if one end of the edge $I_i$ is $O_k$.

Throughout this paper, we shall denote 
\begin{equation*}
\tilde{q}^\e_t= q^\e_{\e t}, \ \ \ \ \tilde{X}^\e_t= X^{\e}_{\e t}.
\end{equation*}

\subsection{Generalized Differential Operator}
In the proof of Theorem \ref{main}, we will need to rewrite each operator $\mathcal{L}_i$ in the form of generalized differential operator. That is we want to find two measures $u_i$ and $v_i$ such that 
\begin{equation*}
\mathcal{L}_i f= \frac{d}{dv_i} (\frac{df}{du_i}).
\end{equation*}
Let $u_i'$ and $v_i'$ be the Radon-Nikodym derivative of $u_i$ and $v_i$ with respect to the Lebesgue measure respectively. For a reason which will be clear lator, we want to choose 
\begin{equation*}
u_i'(x)= (\oint_{C_i(x)} \frac{[(\nabla H)^t \sigma \sigma^* \nabla H](u)}{|g(u)|} dl)^{-1}
\end{equation*}
and
\begin{equation*}
v_i'(x)= \oint_{C_i(x)} \frac{dl}{|g(u)|}.
\end{equation*}
However, with this choice of $u_i$ and $u_i$, $d/du_i d/v_i$ does not generally equal to $\mathcal{L}_i$. In fact
\begin{equation*}
\begin{aligned}
&\frac{d}{dv_i(x)}(\frac{d}{d u_i(x)}f(x)) \\
=& \frac{d}{dx}(f'(x) \cdot \oint_{C_i(x)} \frac{[(\nabla H)^t \sigma \sigma^* \nabla H](u)}{|g(u)|} dl) \cdot \frac{dx}{d v_i(x)}\\
=& f''(x) \oint_{C_i(x)}\frac{[(\nabla H)^t \sigma \sigma^* \nabla H](u)}{|g(u)|} dl \cdot (\oint_{C(x)} \frac{dl}{|g(u)|})^{-1} + \\
&+ f'(x) \frac{d}{dx} \oint_{C_i(x)}\frac{[(\nabla H)^t \sigma \sigma^* \nabla H](u)}{|g(u)|} dl  \cdot (\oint_{C(x)} \frac{dl}{|g(u)|})^{-1}\\
=& A_i(x)f''(x) + f'(x) \frac{d}{dx} \oint_{C_i(x)}\frac{[(\nabla H)^t \sigma \sigma^* \nabla H](u)}{|g(u)|} dl  \cdot (\oint_{C(x)} \frac{dl}{|g(u)|})^{-1}.\\
\end{aligned}
\end{equation*} 
To calculate
\begin{equation*}
\frac{d}{dx} \oint_{C_i(x)}\frac{[(\nabla H)^t \sigma \sigma^* \nabla H](u)}{|g(u)|} dl, 
\end{equation*}
by equation \eqref{invariant measure},
\begin{equation*}
\begin{aligned}
\frac{d}{dx} \oint_{C_i(x)}\frac{[(\nabla H)^t \sigma \sigma^* \nabla H](u)}{|g(u)|} dl 
=&\frac{d}{dx} \oint_{C_i(x)} \frac{\sigma \sigma^* \nabla H}{|g|} \cdot \frac{\nabla H}{|\nabla H|} |\nabla H|dl\\
=&\frac{d}{dx} \oint_{C_i(x)} \frac{\sigma \sigma^* \nabla H}{a} \cdot \vec{\nu} dl\\
\end{aligned}
\end{equation*}
where $\vec{\nu}$ is the unit normal vector of $\partial D_i(x)$. Apply Divergence Theorem, we have
\begin{equation*}
\oint_{C_i(x)} \frac{\sigma \sigma^* \nabla H}{a} \cdot \vec{\nu} dl= \oint_{D_i(x)} \diver(\frac{\sigma \sigma^* \nabla H}{a}) du.
\end{equation*}
To deal with 
\begin{equation*}\label{eq12}
\frac{d}{dx}\oint_{D_i(x)} \diver(\frac{\sigma \sigma^* \nabla H}{a}) du,
\end{equation*}
we have the following lemma
\begin{Lemma}\label{l15}
If we assume that $|\nabla H| \geq c_0 >0$, for any $f \in \textbf{C}^1 (\mathbb{R}^2)$, we have
\begin{equation}\label{eq l15}
\frac{d}{dx}\int_{D_i(x)} f(u) du = \oint_{C_i(x)} f(u) \frac{dl}{|\nabla H(u)|}.
\end{equation}
\end{Lemma}
\begin{proof}
Let $z_t$ be the solution to the following ordinary differential equation
\begin{equation*}
\begin{cases}
dz_t= \frac{\nabla H(z_t)}{|\nabla H(z_t)|^2} dt\\
z_0= \theta \in C_i(x).
\end{cases}
\end{equation*}
Then 
\begin{equation*}
dH(z_t)= \nabla H(z_t) dz_t =\nabla H(z_t) \cdot \frac{\nabla H(z_t)}{|\nabla H(z_t)|^2} dt =dt,
\end{equation*}
which means $H$ can be served as the time and
\begin{equation*}
H(z_t)= t=H,
\end{equation*}
\begin{equation*}
d|z_t|=\frac{1}{|\nabla H(z_t)|}dt=\frac{dH}{|\nabla H(z_t)|}.
\end{equation*}
Let $dl$ denote the unit length on the level set $C_i(x)$, and notice that $z_t$ is orthogonal to the normal vector of the curve $C_i(x)$, $dl \cdot d|z_t|$ is the Lebesgue measure on $\mathbb{R}^2$.
\begin{equation*}
\int_{D_i(x)} f(u) du= \int_{D_i(x)} f(u) dl d|z_t|= \int_{D_i(x)} f(u) \frac{dl}{|\nabla H(z_t)|} dH.
\end{equation*}
The Lemma follows easily by
\begin{equation*}
\int_{D_i(x)} f(u) \frac{dl}{|\nabla H(z_t)|} dH = \int_0^x \int_{C_i(H)} f(u) \frac{dl}{|\nabla H(u)|} dH.
\end{equation*}
\end{proof}
Now we apply equation \eqref{eq l15} and we get
\begin{equation*}
\frac{d}{dx}\oint_{D_i(x)} \diver(\frac{\sigma \sigma^* \nabla H}{a}) du= \oint _{C_i(x)} \diver(\frac{\sigma \sigma^* \nabla H}{a})\frac{dl}{|\nabla H|}.
\end{equation*}
Moreover,
\begin{equation*}
\begin{aligned}
\diver({\frac{\sigma \sigma^* \nabla H}{a}})
=& \sum_{i} \frac{\partial}{\partial x_i} (\sigma \sigma^* \nabla H)_i a^{-1} + (\sigma \sigma^* \nabla H)_i \frac{\partial}{\partial x_i} a^{-1}\\
=& \sum_{i}\frac{\partial}{\partial x_i} (\sum_{k,j} \sigma_{ik} \sigma_{jk} \frac{\partial H}{\partial x_j} )a^{-1} + (\sigma \sigma^* \nabla H)_i \frac{\partial}{\partial x_i} a^{-1}\\
=& (\sum_{i} \sum_{k,j} \sigma_{ik}\sigma_{jk} \frac{\partial^2 H}{\partial x_i \partial x_j}+ \sum_{i}\sum_{j,k} \frac{\partial}{\partial x_i}(\sigma_{ik} \sigma_{jk}) \frac{\partial H}{\partial x_j}a^{-1} +\sum_{i}(\sigma \sigma^* \nabla H)_i \frac{\partial}{\partial x_i} a^{-1}\\
=& (\sum_{i,j} \sum_{k} \sigma_{ik}\sigma_{jk} \frac{\partial^2 H}{\partial x_i \partial x_j} + 2b \cdot \nabla H )a^{-1}+ \diver(\frac{\sigma\sigma^*}{a}) \cdot \nabla H-\frac{2b}{a} \cdot \nabla H\\
=& 2 \mathcal{L}_0 H a^{-1} + [\diver(\frac{\sigma\sigma^*}{a})-\frac{2b}{a}] \cdot \nabla H.\\
\end{aligned}
\end{equation*}
Therefore
\begin{equation*}
\begin{aligned}
\frac{d}{dx}\oint_{D_i(x)} \diver(\frac{\sigma \sigma^* \nabla H}{a}) du=& 2\oint_{C_i(x)} \mathcal{L}_0 H \frac{dl}{a|\nabla H|}+\oint_{C_i(x)}  [\diver(\frac{\sigma\sigma^*}{a})-\frac{2b}{a}] \nabla H \frac{dl}{|\nabla H|}\\
=& 2\oint_{C_i(x)} \mathcal{L}_0 H \frac{dl}{|g(x)|} + \oint_{C_i(x)}  [\diver(\frac{\sigma\sigma^*}{a})-\frac{2b}{a}] \cdot \vec{\nu} dl.
\end{aligned}
\end{equation*}
If we apply the Divergence Theorem again, we get
\begin{equation*}
 \oint_{C_i(x)}  [\diver(\frac{\sigma\sigma^*}{a})-\frac{2b}{a}] \cdot \vec{\nu} dl= \int_{D_i(x)} \diver [\diver(\frac{\sigma\sigma^*}{a})-\frac{2b}{a}] du=2 \int_{D_i(x)}\mathcal{L}_0^* a^{-1} du,
\end{equation*}
where $\mathcal{L}_0^*$ is the formal adjoint of the operator $\mathcal{L}_0$. Therefore
\begin{equation}\label{eq5}
\frac{d}{dx} \oint_{C_i(x)}\frac{(\nabla H)^t \sigma \sigma^* \nabla H(u)}{|g(u)|} dl = 2B_i(x) (\oint_{C_i(x)} \frac{dl}{|g(u)|}) + 2\int_{D_i(x)}  \mathcal{L}_0^*a^{-1} du
\end{equation}
and the following Theorem follows.
\begin{Theorem}\label{t8}
Let $H$ and $f$ satisfy the condition in Lemma \ref{l15}, then
\begin{equation*}
\frac{d}{dv_i(x)}(\frac{d}{d u_i(x)}f(x)) = 2 \mathcal{L}_i f(x) + \frac{2}{T_i(x)}\int_{D_i(x)}  \mathcal{L}_0^* a^{-1} du f'(x).
\end{equation*}
\end{Theorem}

\subsection{Apriori Estimate}
Consider the stopping time
\begin{equation}\label{stopping time}
T^{\e}_q(H_0):= \inf \{t ; H(q^\e_t) \geq H_0  \}.
\end{equation}
We have the following Lemma.
\begin{Lemma}\label{lemma prior estimate}
Under Hypothesis \ref{h1} and \ref{h2}, for any fixed $T>0$, and arbitrary $\eta>0$, there exists a constant $H_0$ such that
\begin{equation*}
\mathbb{P}[T^\e_q(H_0) <T]<\eta. 
\end{equation*}
\end{Lemma}

\begin{proof}
Recall that $x^\e_t=H(q^\e_t)$ satisfies the following stochastic differential equation
\begin{equation*}
dx^\e_t= \mathcal{L}^\e H(q^\e_t) dt + \mathcal{R}^\e H (q^\e_t) dW_t
\end{equation*}
where $\mathcal{L}^\e=\mathcal{L}_0+ \mathcal{L}_0^\e$, $\mathcal{R}^\e=\mathcal{R}_0+ \mathcal{R}_0^\e$. 
By our assumption that $\nabla H$ is of linear growth, $b$, $b^\e$, $\sigma$, $\sigma^\e$ are all Lipschitz continuous. There exists a  constant $C$, such that
\begin{equation*}
\mathcal{L}^\e H(x) \leq C(1+|x|^2),
\end{equation*}
\begin{equation*}
\mathcal{R}^\e H(x) \leq C(1+|x|^2).
\end{equation*}
Therefore
\begin{equation*}
\mathbb{E}[x^\e_t] \leq x^\e_0 + C \int_0^t 1+  \mathbb{E}[|q^\e_s|^2] ds,
\end{equation*}
and since $H(x) \geq a|x|^2$ for $x$ large enough, there exists a constant $C$ such that
\begin{equation*}
\mathbb{E}[x^\e_t] \leq C(1+T) + C \int_0^T \mathbb{E}[x^\e_s] ds, \ \ \ \ \ t \leq T.
\end{equation*}
If we apply the Gronwall's inequality, this implies
\begin{equation*}
\mathbb{E}[x^\e_t] \leq C(1+T) e^{Ct},\ \ \ \ t \leq T.
\end{equation*}
Also,
\begin{equation*}
\sup_{0\leq t \leq T} \int_0^t \mathcal{L}^\e H(q^\e_s) ds \leq  \int_0^T C(1+ |q^\e_s|^2) ds.
\end{equation*}
So that,
\begin{equation}\label{eq1}
\begin{aligned}
\mathbb{P}[\sup_{0\leq t \leq T} \int_0^t \mathcal{L}^\e H(q^\e_s) ds  \geq R] 
&\leq \frac{C}{R} \mathbb{E}[\int_0^T (1+ |q^\e_s|^2) ds] \leq \frac{C}{R} \mathbb{E}[\int_0^T (1+ \frac{1}{a}H(q^\e_s)) ds]\\
&\leq \frac{C}{R} \int_0^T C(1+ \frac{1}{a}C(1+s)e^{Cs} ds \leq \frac{C_T}{R}.
\end{aligned}
\end{equation}
Now pick $R$ such that $C_T/R < \eta$, and $H_0>x_0+R$
\begin{equation*}
\begin{aligned}
\mathbb{P}[\sup_{0\leq t \leq T} x^\e_t \geq H_0] 
&=\mathbb{P}[\sup_{0 \leq t \leq T} \int_0^t \mathcal{L}^\e H (q^\e_s) ds + \int_0^t \mathcal{R}^\e H (q^\e_s) dW_s \geq H_0-x_0]\\
&\leq \mathbb{P}[\sup_{0 \leq t \leq T} \int_0^t \mathcal{L}^\e H (q^\e_s) ds + \int_0^t \mathcal{R}^\e H(q^\e_s) dW_s \geq H_0-x_0 ; \sup_{0\leq t \leq T} \int_0^t \mathcal{L}^\e H(q^\e_s) ds \\ &\ \ \ \ \ \ \  \leq R]+ \mathbb{P}[\sup_{0\leq t \leq T} \int_0^t \mathcal{L}^\e H(q^\e_s) ds>R]\\
&\leq \mathbb{P}[\sup_{0 \leq t \leq T} \int_0^t \mathcal{R}^\e H(q^\e_s) dW_s \geq H_0-x_0 -R] + \mathbb{P}[\sup_{0\leq t \leq T} \int_0^t \mathcal{L}^\e H (q^\e_s) ds>R].\\
\end{aligned}
\end{equation*}
Due to \eqref{eq1}, the second term above is smaller than $\eta/2$ by our choice of $R$. For the first term, we have
\begin{equation}\label{eq2}
\begin{aligned}
     & \mathbb{P}[\sup_{0 \leq t \leq T} \int_0^t \mathcal{R}^\e H (q^\e_s) dW_s \geq H_0-x_0 -R]\\
\leq& \frac{1}{(H_0-x_0-R)^2} \mathbb{E}[\int_0^T |\mathcal{R}^\e H (q^\e_s)|^2 ds]\\
\leq& \frac{C}{(H_0-x_0-R)^2} \mathbb{E}[\int_0^T C (1+|q^\e_t|^4)dt]\\
\leq& \frac{C}{(H_0-x_0-R)^2} \mathbb{E}[\int_0^T (1+|H(q^\e_t)|^2)dt]\\
\leq& \frac{C_T}{a^2(H_0-x_0-R)^2} + \frac{C}{(H_0-x_0-R)^2} \int_0^T \mathbb{E}[|x^\e_t|^2]dt.
\end{aligned}
\end{equation}
Now,
\begin{equation*}
\begin{aligned}
\mathbb{E}[|x^\e_t|^2] 
&\leq 3 x_0^2 + 3 \mathbb{E}[(\int_0^t C(1+ H(q^\e_s))ds)^2] + 3 \mathbb{E}[\int_0 ^t C(1+ |x^\e_s|^2) ds]\\
&\leq 3 x_0^2 + 3 \mathbb{E}[\int_0^t C(1+|x^\e_s|^2) ds T] + 3 \mathbb{E}[\int_0^t |x^\e_s|^2 ds] + 3CT\\
&\leq C(x_0,T) + C(T) \mathbb{E}[\int_0^t|x^\e_s|^2 ds],
\end{aligned}
\end{equation*}
so that
\begin{equation}\label{eq3}
\mathbb{E}[|x^\e_t|^2] \leq C(x_0 , T) e^{C(T^{2}+T)}= C(T,x_0).
\end{equation}
Therefore, we can pick $H_0$ large enough in \eqref{eq2} so that
\begin{equation*}
\mathbb{P}[\sup_{0\leq t \leq T} x^\e_t \geq H_0] < \eta.
\end{equation*}
\end{proof}

\subsection{Lipschitz Continuity}
\begin{Lemma}
Under Hypothesis \ref{h1} and \ref{h2}, for any continous function $f$, $A_i$, $B_i$ and $\mathcal{L}_0 f$ are Liptichitz continuous in $D(H_1,H_2) \subseteq \mathbb{R}^2 $ where $H_1$ and $H_2$ are inside the interior of some edge $I_i$.
\end{Lemma}
\begin{proof}
The tool we use here Lemma 1.1 in Chapter 8 of \cite{fw}. We first calculate the derivative of $T_i$. 
\begin{equation*}
\begin{aligned}
&\frac{d}{dx} T_i(x)=\frac{d}{dx} \oint_{C_i(x)} \frac{dl}{|g(u)|}
= \frac{d}{dx}  \oint_{C_i(x)} \frac{|\nabla H(u)|}{|g(u)||\nabla H(u)|} dl\\
=& \oint_{C_i(x)} \nabla (\frac{1}{|g(u)||\nabla H(u)|})\cdot \frac{\nabla H(u)}{|\nabla H(u)|} + \frac{1}{|g(u)||\nabla H(u)|} \cdot \frac{\Delta H(u)}{|\nabla H(u)|}dl.
\end{aligned}
\end{equation*}
Let 
\begin{equation*}
\bar{A}_i(x)= \oint_{C_i(x)}{\mathcal{R}_0[H](u) \mathcal{R}_0[H](u)^*\frac{dl}{|g(u)|}},\ \ \ \bar{B}_i(x)=\oint_{C_i(x)} {\mathcal{L}_0[H](u)  \frac{dl}{|g(u)|}}.
\end{equation*}
By equation (\ref{eq5})
\begin{equation}\label{eq6}
\frac{d}{dx} \bar{A}_i(x)=  2\bar{B}_i(x)+ 2\int_{D_i(x)}  \mathcal{L}_0^*[a^{-1}] du.
\end{equation}
Next, we apply Lemma 1.1 in Chapter 8 of \cite{fw} again to calculate $\bar{B}_i'(x)$
\begin{equation*}
\frac{d}{dx} \bar{B}_i(x) =  \oint_{C_i(x)} \nabla (\frac{\mathcal{L}_0[H](u) }{|g(u)||\nabla H(u)|})\cdot \frac{\nabla H(u)}{|\nabla H(u)|} + \frac{\mathcal{L}_0[H](u)  }{|g(u)||\nabla H(u)|} \cdot \frac{\Delta H(u)}{|\nabla H(u)|}dl.
\end{equation*}
Also we apply Lemma \ref{l15} to calculate the derivative of the residue in (\ref{eq6})
\begin{equation*}
\begin{aligned}
&\frac{d}{dx} \int_{D_i(x)}  \mathcal{L}_0^*[a^{-1}] du = \oint_{C_i(x)}   \mathcal{L}_0^*[a^{-1}] \frac{dl}{|\nabla H(x)|}.
\end{aligned}
\end{equation*}
By our assumptions, it can be easily checked that $T_i'(x)$ is both bounded and bounded below above 0, and both $\bar{A}_i'(x)$ and $\bar{B}_i'(x)$ are bounded. Therefore ${A}_i'(x)$ and $B_i'(x)$ are bounded in $D_i$, which implies their Lipschitz continuity. The Lipschitz conitnuity for $\mathcal{L}_0 f$ is obvious.
\end{proof}

\section{Proof of Theorem \ref{main}}
We first define the following sequence of stopping times. In the definition below, $\d$ and $\d'$ are to be determined later.
\begin{Definition}\label{d2}
Let $S$ be the set of integers $k \leq m$, such that $O_k$ is a saddle point. For any $\e>0$ and $0<\d'<\d$, let $0= \tau^\e_0 \leq \sigma^\e_1 \leq \tau^\e_1 \leq ... \leq \sigma^\e_{n} \leq \tau^\e_{n} \leq ... $ be a sequence of stopping times with
\begin{equation*}
\sigma^\e_n=  \inf \{t\geq \tau^\e_{n-1};  q^\e_t  \notin \bigcup_{i \in S} D_i(\pm \d)\} \wedge T^\e_q(H_0) 
\end{equation*}
and
\begin{equation*}
\tau^\e_n=  \inf \{t\geq \sigma^\e_{n};  q^\e_t  \notin \bigcup_{i \in S} C_i(\pm \d')\} \wedge T^\e_q(H_0). 
\end{equation*}
Where $T^\e_q(H_0)$ is the stopping time defined in \eqref{stopping time}. Moreover, we define 
\begin{equation*}
\tau^\e_k(\pm \d):= \inf\{ t\geq 0 ; q^\e_t \notin D_k(\pm \d)\}.
\end{equation*}
\end{Definition}

Recall that we denote by $L$ the infinitesimal operator of the process $(x_t, i_t)$ on $\Gamma$. If the Poisson problem $(\a I - L)f= u$ has a unique solution then this solution has the representation
\begin{equation*}
f(x,i)=\mathbb{E}_{(x,i)}\int_0^{+\infty} e^{-\a t} u(x_t,i_t)dt,
\end{equation*} 
Replacing $u$ by $(\a I - L)f$, gives 
\begin{equation*}
f(x,i)= \mathbb{E}_{(x,i)}\int_0^{+\infty} e^{-\a t}(\a I - L)[f](x_t,i_t)dt.
\end{equation*}
If we can prove that for all $u \in D(L)$,
\begin{equation*}
\lim_{\e \to 0} \mathbb{E}_{(x,i)}\int_0^{+\infty} e^{-\a t}(\a f - Lf)(x^\e_t,i^\e_t)dt = f(x,i),
\end{equation*}
then the tightness of the family $\{x^\e_t, i^\e_t\}_{\e>0}$, Prokhorov theorem and its corollary, and the fact that the range of the operator $\a I - L $ uniquely determines a measure guarantee that $\{(x^\e_t, i^\e_t)\}_{\e >0}$ converges weakly to $(x_t, i_t)$ as $\e \to 0$ in $\textbf{C}([0,+\infty);\Gamma)$. 

By Freidin and Wenztell's procedure, the tightness of the measure on $\textbf{C}([0,+\infty);\Gamma)$ follows easily from apriori estimate proved in Lemma \ref{lemma prior estimate}. Hence it is sufficient to prove that, for any $\eta>0$ and $H_0>0$, there exists $\e_0$ such that for all $0<\e <\e_0$
\begin{equation*}
|\mathbb{E}_{(x,i)} \int_0^{T^\e_q(H_0)} e^{-\a t}(\a I - L)f(x_t,i_t)dt - f(x,i)| < \eta
\end{equation*}
which conatins two parts: the part where $(x^\e_t,i^\e_t)$ remains in the same edge and the part where $(x^\e_t,i^\e_t)$ approaches an inner vertex of the graph. We have
\begin{equation*}\label{eq9}
\begin{aligned}
&\mathbb{E}_{(x,i)}[\int_0^{T^\e_q(H_0)} e^{-\a t}(\a f - L f)(x^\e_t,i^\e_t)dt - f(x,i)]\\
=&\mathbb{E}_{(x,i)}[\sum_{n=0}^{\infty} e^{-\a \sigma^\e_{n+1}} f(x^\e_{\sigma^\e_{n+1}},i^\e_{\sigma^\e_{n+1}}) - e^{-\a \tau_{n}} f(x^\e_{\tau^\e_{n}},i^\e_{\tau^\e_{n}}) + \int_{\tau^\e_{n}}^{\sigma^\e_{n+1}} e^{-\a t}(\a I - L)f(x^\e_t,i^\e_t) dt]\\
&+ \mathbb{E}_{(x,i)}[\sum_{n=1}^{\infty} e^{-\a \tau^\e_{n}} f(x^\e_{\tau^\e_{n}},i^\e_{\tau^\e_{n}}) - e^{-\a \sigma^\e_{n}} f(x^\e_{\sigma^\e_{n}},i^\e_{\sigma^\e_{n}}) + \int_{\sigma^\e_{n}}^{\tau^\e_{n}} e^{-\a t}(\a I - L)f(x^\e_t,i^\e_t) dt]\\
=&\mathbb{E}_{(x,i)}[\sum_{n=0}^{\infty} \mathbb{E}_{(x,i)}[ e^{-\a \sigma^\e_{n+1}} f(x^\e_{\sigma^\e_{n+1}},i^\e_{\sigma^\e_{n+1}}) - e^{-\a \tau_{n}} f(x^\e_{\tau^\e_{n}},i^\e_{\tau^\e_{n}}) + \int_{\tau^\e_{n}}^{\sigma^\e_{n+1}} e^{-\a t}(\a I - L)f(x^\e_t,i^\e_t) dt | \mathcal{F}_{\tau_n^\e}]]\\
&+ \mathbb{E}_{(x,i)}[\sum_{n=1}^{\infty} \mathbb{E}_{(x,i)}[e^{-\a \tau^\e_{n}} f(x^\e_{\tau^\e_{0},i^\e_{\tau^\e_{0}}}) - e^{-\a \sigma^\e_{n}} f(x^\e_{\sigma^\e_{n}},i^\e_{\sigma^\e_{n}}) + \int_{\sigma^\e_{n}}^{\tau^\e_{n}} e^{-\a t}(\a I - L)f(x^\e_t,i^\e_t) dt | \mathcal{F}_{\sigma_n^\e}]]\\
=&\mathbb{E}_{(x,i)}[\sum_{n=0}^{\infty}e^{-\a \tau^\e_{n}} \mathbb{E}_{(x^\e_{\tau^\e_n},i^\e_{\tau^\e_n})}[ e^{-\a \sigma^\e_{1}} f(x^\e_{\sigma^\e_{1}},i^\e_{\sigma^\e_{1}}) -  f(x^\e_{0},i^\e_{0}) + \int_{0}^{\sigma^\e_{1}} e^{-\a t}(\a I - L) f(x^\e_t,i^\e_t) dt ]]\\
&+ \mathbb{E}_{(x,i)}[\sum_{n=1}^{\infty} e^{-\a \sigma^\e_{n}}\mathbb{E}_{(x^\e_{\sigma^\e_{n}},i^\e_{\sigma^\e_{n}})}[e^{-\a \tau^\e_{1}} f(x^\e_{\tau^\e_{1}},i^\e_{\tau^\e_{1}}) - f(x_0^\e,i^\e_0) + \int_{0}^{\tau^\e_{1}} e^{-\a t}(\a I - L)f(x^\e_t,i^\e_t) dt ]].\\
\end{aligned}
\end{equation*}
If we define
\begin{equation*}\label{estimation near the interior vertex}
\Phi^\e_{1}(x,i)=\mathbb{E}^\e_{(x, i)}[e^{- \a \sigma^\e_1}f(x^\e_{\sigma^\e_1}, i^\e_{\sigma^\e_1}) + \int_{0}^{\sigma^\e_1} e^{- \a t} (\a I - L)f(x^\e_t, i^\e_t)dt] -  f(x,i),
\end{equation*}
and
\begin{equation*}\label{estimation inside the edge}
\Phi^\e_2(x,i)= \mathbb{E}^\e_{(x, i)}[e^{-\a \tau^\e_{1}} f(x^\e_{\tau^\e_{1}},i^\e_{\tau^\e_{1}}) + \int_{0}^{\tau^\e_{1}} e^{-\a t}(\a I - L)f(x^\e_t,i^\e_t) dt]- f(x, i).
\end{equation*}
We have
\begin{equation*}
\begin{aligned}
&|\mathbb{E}_{(x,i)}[\int_0^{T^\e_q(H_0)} e^{-\a t}(\a f - Lf)(x^\e_t,i^\e_t)dt - f(x,i)]| \\
\leq& \mathbb{E}_{(x,i)} \sum_{n=0}^{\infty}e^{-\a \tau^\e_{n}}  \sup_{(x,i) \in \bigcup_{i \in S} C_i(\pm \d')} |\Phi^\e_1(x,i)| +
\mathbb{E}_{(x,i)} \sum_{n=1}^{\infty}e^{-\a \sigma^\e_{n}}  \sup_{(x,i) \in \bigcup_{i \in S} C_i(\pm \d)} |\Phi^\e_2(x,i)|.
\end{aligned}
\end{equation*}
Follow the same procedure in \cite{fw} and get, for sufficiently small $\e$,
\begin{equation*}
\mathbb{E}_{(x,i)}[\sum_{n=1}^{\infty}e^{-\a \sigma^\e_{n}} ] \leq \mathbb{E}_{(x,i)}[\sum_{n=0}^{\infty}e^{-\a \tau^\e_{n}} ]  \leq C \d^{-1},
\end{equation*}
so that for sufficiently small $\e$,
\begin{equation*}
\begin{aligned}
&|\mathbb{E}_{(x,i)}[\int_0^{T^\e_q(H_0)} e^{-\a t}(\a f - Lf)(x^\e_t,i^\e_t)dt - f(x,i)]| \\
\leq&C\d^{-1}  (\sup_{(x,i) \in \bigcup_{i \in S} C_i(\pm \d')} |\Phi^\e_1(x,i)| +  \sup_{(x,i) \in \bigcup_{i \in S} C_i(\pm \d)} |\Phi^\e_2(x,i)|).
\end{aligned}
\end{equation*}
In order to study the term $\Phi_2$, we need the following result.
\begin{Lemma}\label{convergence in edge}
Let $\beta_t$ be a one dimensional standard Brownian motion, and let $x_t$ be the solution to the following equation
\begin{equation*}
d{x}_t=B_i({x}_t)dt + A_i^{\frac{1}{2}}({x}_t)d\beta_t,
\end{equation*}
Assume $H_1<H_2$, and let $\tau^\e_{H_1,H_2}$ be the stopping time when $q^\e_t$ leaves the region $D_i(H_1 ,H_2)$. If either one of these three cases holds,
\begin{enumerate}
\item $I_i$ is an edge such that both of its vertex are interior vertex, and $H_1<H_2$ are any fixed values belonging to the interior of the interval.
\item $I_i$ is an edge such that one of its vertex is an interior vertex while the other vertex is an exterior vertex, where the Hamiltonian takes the value $H_1$, $H_2$ belonging to the interior of the interval. 
\item $I_i$ is an edge that has only one vertex, and $H_1<H_2$ are any fixed values belonging to the interior of the interval.
\end{enumerate}
Then under Hypothesis \ref{h1} and \ref{h2}, for every function $f$ on $[H_1,H_2]$ that is three times continuously differentiable and for every $\la>0$,
\begin{equation*}\label{wek convergence in edge}
\lim_{\e \to 0}\mathbb{E}_{x_0}[e^{-\la \tau^\e_{H_1,H_2}} f(x^\e_{\tau^\e_{H_1,H_2}}) - \int_0^{\tau^\e_{H_1,H_2}} e^{-\la s}(-\la I + \mathcal{L}_i)[f](x^\e_{s}) ds ] = f(x_0)
\end{equation*}
uniformly with respect to $x_0 \in [H_1,H_2]$.
\end{Lemma}

In order to study the term $\Phi_1$, we need the following result.
\begin{Lemma}\label{not sticky lemma}
Let $O_k$ be an interior vertex, then for every positive $\a$ and $\kappa$, there exists $\d_0>$ such that for all $0< \d < \d_0$ and sufficiently small $\e$, 
\begin{equation*}
\sup_{z\in D_k(\pm \d)} \mathbb{E}_{z}^\e[\int_0^{\tau^\e_k(\pm \d)} e^{-\a t} dt] < \kappa \d,
\end{equation*}
and there exists constant C such that
\begin{equation*}
\sup_{q \in D_k(\pm \d)}\mathbb{E}^\e_z[\tau^\e_k(\pm \d) ] \leq C \d^2 \ln(\d).
\end{equation*}
(see remark on page 310 in \cite{fw})
\end{Lemma}
We can use Lemma 3.5 in \cite{fw} and apply the change of random time to prove this Lemma easily. Now if we take 
\[
\kappa= \frac{\eta}{\a ||f||_{\infty} +||\a f - L f||_{\infty}}
\]
in Lemma \ref{not sticky lemma}, we have
\begin{equation*}\label{phi1 decomposition}
\begin{aligned}
|\Phi^\e_1(x,i)| 
\leq& |\mathbb{E}_{(x,i)}[f(x^\e_{\sigma^\e_1}, i^\e_{\sigma^\e_1}) - f(x,i)]| + (\a ||f||_{\infty} +||\a f - L f||_{\infty}) \cdot 
\frac{\d}{\a ||f||_{\infty} +||\a f - L f||_{\infty}} \eta\\
\leq&  \sum_{I_j \sim O_k} |\mathbb{E}_{q_0}[ f_j(H(q^\e_{\sigma^\e_1})) - f_i(H(q)); q^\e_{\sigma^\e_1} \in I_j]|+\d \eta\\
=& \sum_{I_j \sim O_k} |f_j(H(O_k)+ \d)-f_i(x)| \cdot \mathbb{P}_{q_0}[q^\e_{\sigma^\e_1} \in I_j] + \d \eta.
\end{aligned}
\end{equation*}
We will prove the following Lemma in Section \ref{gluing condition} to handle the term 
\[
\sup_{ q_0 \in  \bigcup_{I_j \sim O_k} C_j(\pm \d')\}}\mathbb{P}_{q_0}[q^\e_{\sigma^\e_1} \in I_j]
\]
The following lemma will be proved in section 5.
\begin{Lemma}\label{lemma gluing condition}
For every $\kappa>0$ there exists a positve $\d_0>0$ such that for $0<\d<\d_0$ there exists $\d'_0=\d'_0(\d)$ such that for sufficiently small $\e$,
\begin{equation*}
\sup_{x \in \bar{D}_k(\pm \d'_0)}|\mathbb{P}_x[\hat{q}^\e_{\tau^\e_k(\pm\d)} \in C_{ki}(\d)]-p_{ki}|<\kappa,
\end{equation*}
where $p_{ki}$ is the constant defined in \eqref{gluing probability}.
\end{Lemma}
If we take $\kappa < \eta$, then for sufficiently small $\e$,
\begin{equation*}
\begin{aligned}
\Phi^\e_1(x,i)  
\leq& \sum_{I_j \sim O_k} |f_j(H(O_k)+ \d)-f_i(x)| \cdot \mathbb{P}_{q_0}[q^\e_{\sigma^\e_1} \in I_j] + \d \eta\\
\leq& \sum_{I_j \sim O_k}  ||f_j(H(O_k)+ \d)-f_i(x)|- \d D_{jk}f(x)| \cdot |\mathbb{P}_{q_0}[q^\e_{\sigma^\e_1} \in I_j]-p_{kj}| + \\
&\sum_{I_j \sim O_k} ||f_j(H(O_k)+ \d)-f_i(x)|- \d D_{jk}f(x)|\cdot  p_{kj} \\ 
&\ \ \ \ \ \ \ \ \ \ \ \ \ \ \ \ \ \ \ \ \ \ \ \ \ \ \ \ \ \ \ \ \ \ \ \ \ \ \ \ \ \ \ \ \ \ \ \ \  \ \ \ \ \ \ \ \ \ +\sum_{I_j \sim O_k} \d |D_{jk}f(x)| \cdot |\mathbb{P}_{q_0}[q^\e_{\sigma^\e_1} \in I_j]-p_{kj}| \\
\leq& C(\kappa \d^2 + \d^2 + \d \kappa) \leq C\d \eta.
\end{aligned}
\end{equation*}

To this point, we first pick $\kappa < \eta$ and then determine $\d$ and $\d' $according to Lemma \ref{lemma gluing condition}. After that, we let $\e$ sufficiently small so that Lemma \ref{lemma gluing condition} holds for $\kappa< \eta$ and Lemma \ref{convergence in edge} holds for
 \[
 \sup_{(x,i) \in \bigcup_{i \in S} C_i(\pm \d)} |\Phi^\e_2(x,i)| \leq \d \eta,
\]
Therefore
\begin{equation*}
\mathbb{E}_{(x,i)}[\int_0^{T^\e_q(H_0)} e^{-\a t}(\a f - Lf)(x^\e_t,i_t)dt - f(x,i)] \leq \eta + C\eta,
\end{equation*}
which proves the result.

\section{Weak Convergence Inside the Edge} \label{Weak Convergence Inside the Edge}
We first prove the weak convergence of $x^\e_t$ inside an edge of finite length. For $H_1<H_2<\infty$, recall that $\tau^\e_{H_1,H_2}$ is the stopping time when the process $q^\e_t$ leaves the region $D_i(H_1, H_2)$. Noticed that the coefficients are all bounded in $D_i(H_1, H_2)$. Let $T^\e(z)$ be the period of $X^\e_t$ starting at $z$. Suppose $\underline{C} \e  \leq T^\e(z) \leq \bar{C} \e$ for all $z \in D_i(H_1,H_2)$. Where we should remark that the constant $\underline{C} $ and $\bar{C}$ depends on the constant $H_1$ and $H_2$. In this section, we consider the process $q^\e_t$ starting from $q_0 \in D_i(H_1,H_2)$ and denote $\tau^\e_T= \tau^\e_{H_1,H_2} \wedge T$. We will first prove a weaker version of Lemma \ref{convergence in edge}.

\begin{Lemma}\label{convergence in edge weaker version}
Let $\beta_t$ and $x_t$ be defined as in Lemma \ref{convergence in edge}. Under the same condition as in Lemma \ref{convergence in edge}, for every function $f$ on $[H_1,H_2]$ that is three times continuously differentiable and for every $\la>0$,
\begin{equation*}\label{wek convergence in edge}
\lim_{\e \to 0}\mathbb{E}_{x_0}[e^{-\la \tau^\e_T} f(x^\e_{\tau^\e_T}) - \int_0^{\tau^\e_T} e^{-\la s}(-\la f + \mathcal{L}_i f)(x^\e_{s}) ds ] = f(x_0)
\end{equation*}
uniformly with respect to $x_0 \in [H_1,H_2]$.
\end{Lemma}
In what follows we shall define a sequence of stopping times $T^\e_k$ defined by $T^\e_0=T^\e(q)$, $T_1=T^\e(q^\e_{T^\e_0})+T^\e_0$, ..., $T^\e_k=T^\e(q^\e_{T^\e_{k-1}})+T^\e_{k-1}$,... 
We first define an auxilliary process $\xi^\e_t$ trajectory by trajectory. 
\begin{Definition}
 Let $V^\e_s={\mathcal{R}_0 H}/{|\mathcal{R}_0 H|}(q^\e_s)$, and
\begin{equation*}
\beta^\e_{t}= \int_0^t V^\e_s dW_s.
\end{equation*} 
For every $k \in \mathbb{N}$, we define $\xi^\e_{T_k}=x^\e_{T_k}$ and for $t \in [T^\e_k\wedge \tau^\e_T, T^\e_{k+1} \wedge \tau^\e_T)$, 
\begin{equation*}
\xi^\e_t= \xi_{T^\e_k}+ B_i(x^\e_{T^\e_k})(t-T^\e_k) + A_i^{\frac{1}{2}}(x^\e_{T^\e_k}) (\beta^\e_t-\beta^\e_{T^\e_{k}}).
\end{equation*}
\end{Definition}
\begin{Remark}
We should notice that $\beta^\e_t$ is a standard Brownian motion on $\mathbb{R}$.
\end{Remark}
We should also define the auxilliary process $\tilde{x}^\e_t$ which has the same distribution as the limiting process of the slow motion.
\begin{Definition}
For $t \in [0,\tau^\e_T]$, $\tilde{x}^\e_t$ is the solution of the problem
\begin{equation*}
d\tilde{x}^{\e}_t=B_i(\tilde{x}^{\e}_t)dt + A_i^{\frac{1}{2}}(\tilde{x}^{\e}_t)d\beta^{\e}_t,
\end{equation*}
with the initial condition $x_0=H(q) \in [H_1,H_2]$.
\end{Definition}

\subsection{Closeness of the $q^\e_t$ and $\bar{X}^{\e,k}_s$}
Let $\bar{X}^{\e}_t$ satisfies the equation $d\bar{X}^\e_t= \e^{-1}g(\bar{X}^\e_t) dt $ with the initial condition $\bar{X}^\e_0=q_0$. It can be easily seen that
\begin{equation*}
|q^\e_s - \bar{X}^\e_s| \leq \int_0^s \frac{L}{\e} |q^\e_u - \bar{X}^\e_u| du + \int_0^s C(1+L) |q^\e_u|du + |\int_0^s (\sigma+ \sigma^\e) (q^\e_u) dW_u|.
\end{equation*}

Therefore, for $t<T$,
\begin{equation*}
\begin{aligned}
\mathbb{E}[\sup_{0\leq s \leq t} |q^\e_s- \bar{X}^\e_s|  ^4] 
\leq& \frac{cL^4}{\e^4} \mathbb{E}[(\int_0^t\sup_{0\leq u\leq s} |q^\e_u-\bar{X}^\e_u| ds)^4] + c t^3 \int_0^t C^2(1+L)^2 \mathbb{E}[|q^\e_u|^4] du \\
&+ c \mathbb{E} [|\sup_{0\leq s \leq t } \int_0^s (\sigma+ \sigma^\e)(q^\e_u) dW_u|^4]\\
\leq& \frac{cL^4}{\e^4} \mathbb{E}[\int_0^t \sup_{0\leq u \leq s} |q^\e_u - \bar{X}^\e_u|  ^4 ds \cdot t^3] + c t^3 \int_0^t C^2(1+L)^2 \mathbb{E}[|x^\e_u|^2] du \\
&+ c(\mathbb{E} [\int_0^t \Tr[(\sigma+ \sigma^\e)^*(\sigma+ \sigma^\e)](q^\e_s) ds])^2\\ 
\leq& \frac{Ct^3}{\e^4} \mathbb{E}[\int_0^t \sup_{0\leq u \leq s} |q^\e_u - \bar{X}^\e_u|  ^4 ds ] + C(1+t+t^3)   \int_0^t  \mathbb{E}[|x^\e_u|^2] du. 
\end{aligned}
\end{equation*}
Due to equation (\ref{eq3}), $\mathbb{E}[|x^\e_u|^2] \leq C(T,x_0)$, we have
\begin{equation*}
\mathbb{E}[\sup_{0\leq s \leq t} |q^\e_s- \bar{X}^\e_s|  ^4]  \leq  \frac{Ct^3}{\e^4} \mathbb{E}[\int_0^t \sup_{0\leq u \leq s} |q^\e_u - \bar{X}^\e_u|  ^4 ds ] + C(T,x_0)(t+t^2+t^4).
\end{equation*}
Apply Gronwall's inequality, we have
\begin{equation}\label{eq closeness of trajectory}
\mathbb{E}[\sup_{0\leq s \leq t} |q^\e_s- \bar{X}^\e_s|  ^4] < C(T,x_0)t(1+t+t^3) \exp(\frac{Ct^4}{\e^4}).
\end{equation}

\subsection{$\Delta^k_{b}(t)$ and $\Delta^k_{\sigma}(t)$}
In this subsection, we will introduce two critical term in the estimation of the closeness of the trajectories. For $t \in [T^\e_{k-1} \wedge \tau^\e_T, T^\e_k \wedge \tau^\e_T]$, let
\begin{equation*}
\Delta^{k-1}_b(t)= |\int_{T^\e_{k-1}\wedge \tau^\e_T}^{t} \mathcal{L}_0 H (q^\e_s) ds - (t -T^\e_{k-1}\wedge \tau^\e_T) B_i(x^{\e}_{T^\e_{k-1}\wedge \tau^\e_T})|,
\end{equation*}
and 
\begin{equation*}
\Delta^{k-1}_b(T^\e_k\wedge \tau^\e_T)=\Delta^{k-1}_b.
\end{equation*}
It is easy to see that  $\Delta^{k-1}_b(t)=0$, if $\tau^\e_T< T^\e_{k-1}$. Otherwise,
\begin{equation*}
\begin{aligned}
|\Delta^{k-1}_b(t)|
&= |\int_{T^\e_{k-1}}^{t } \mathcal{L}_0 H (q^\e_s) ds - (t-T^\e_{k-1} \wedge \tau^\e_T) B_i(x^{\e}_{T^\e_{k-1}})|\\
&= |\int_{T^\e_{k-1} }^{t} \mathcal{L}_0 H (q^\e_s) ds -  \frac{t-T^\e_{k-1} }{T^\e(q^{\e}_{T^\e_{k-1}})} \oint_{C_i(x^\e_{T^\e_{k-1}}) } {\mathcal{L}_0(u)\frac{\e dl}{|g(u)|}}|\\
&=|\int_{T^\e_{k-1}}^{t} \mathcal{L}_0[H](q^\e_s) ds - \frac{t-T^\e_{k-1}}{T^\e_k-T^\e_{k-1}}\int_{T^\e_{k-1}}^{T^\e_k } \mathcal{L}_0(\bar{X}^{\e,k}_s)ds|\\
&\leq \int_{T^\e_{k-1}}^{T^\e_k \wedge \tau} |\mathcal{L}_0 H (q^\e_s)-\mathcal{L}_0 H (\bar{X}^{\e,k}_s)|ds + \int_{t}^{T^\e_k} |\mathcal{L}_0 H (\bar{X}^{\e,k}_s)| ds \\
&\ \ \ \ \ \ \ \ \ \ \ \ \ \ \ \ \ \ \ \ \ \ \ \ \ \ \ \ \ \ \ \ \ \ \ \ \ \ \ \ \ \ \ \ \ \ \ \ \ \ \ \ \ \ \ \ \ \  + \frac{T^\e_k-t}{T^\e_k - T^\e_{k-1}} \int_{T^\e_{k-1}}^{T^\e_k} |\mathcal{L}_0 H (\bar{X}^{\e,k}_s)| ds.
\end{aligned}
\end{equation*}
Where $\bar{X}^{\e,k}_s$ satisfies the equation $d\bar{X}^{\e,k}_t= \frac{1}{\e} g(\bar{X}^{\e,k}_t)dt,\  \bar X^{\e,k}_0=q^\e_{T^\e_{k-1}}$. Notice that the last two terms in the summation appears only once in all $\{\Delta^k_b(t)\}_{k \in \mathbb{N}}$. For the first term, we have
\begin{equation*}
\begin{aligned}
\int_{T^\e_{k-1}}^{T^\e_k \wedge \tau^\e_T} |\mathcal{L}_0 H(q^\e_s)-\mathcal{L}_0 H(\bar{X}^\e_s)|ds 
&\leq \int_{T^\e_{k-1}}^{T^\e_k \wedge \tau^\e_T} |\mathcal{L}_0 H|_{Lip} |q^\e_s- \bar{X}^\e_s| ds \\
&\leq C(\mathcal{L}_0, I_i) \int_{T^\e_{k-1}}^{T^\e_k \wedge \tau^\e_T} |q^\e_s- \bar{X}^\e_s| ds.
\end{aligned}
\end{equation*}
Thus, if we set, for $t \in [T^\e_{k-1}\wedge \tau^\e_T, T^\e_k \wedge \tau^\e_T]$ and every $k$ except one, we get
\begin{equation*}
\begin{aligned}
\mathbb{E}[|\Delta^{k-1}_b(t) |^4 | \mathcal{F}_{T^\e_{k-1}}] 
&\leq C\mathbb{E}[(\int_{T^\e_{k-1}}^{T^\e_k \wedge \tau^\e_T} |q^\e_s- \bar{X}^{\e,k}_s| ds)^4 | \mathcal{F}_{T^\e_{k-1}}]\\
&\leq C \mathbb{E}_{q^\e_{T^\e_{k-1}}}[\int_0^{T^\e_0} |q^\e_s-\bar{X}^{\e,k}_s|^4 ds (T^\e_0 \wedge \tau^\e_T)^3 ]\\
&\leq C (\bar{C} \e)^4 \mathbb{E}_{q^\e_{T^\e_{k-1}}}[\sup_{0 \leq s \leq \bar{C} \e} |q^\e_s - \bar{X}^\e_s|^4]. 
\end{aligned}
\end{equation*}

Similarlly, for $t \in [T^\e_{k-1} \wedge \tau^\e_T, T^\e_k \wedge \tau^\e_T]$, let
\begin{equation*} 
\Delta^{k-1}_\sigma(t)= \int_{T_{k-1}\wedge \tau^\e_T}^{t} |A_i^{\frac{1}{2}}(x^\e_{T^\e_{k-1} \wedge \tau^\e_T})V^\e_s- \mathcal{R}_0 H (q^\e_s)|^2 ds
\end{equation*}
and
\begin{equation*}
\Delta^{k-1}_{\sigma}(T^\e_k \wedge \tau^\e_T)=\Delta^{k-1}_{\sigma}.
\end{equation*}
Due to the inequality that $(x-y)^2 \leq 2 |x^2-y^2|$, for $x,y \in \mathbb{R}^+$, we have for $t \in [T^\e_{k-1} \wedge \tau^\e_T, T^\e_k \wedge \tau^\e_T]$
\begin{equation*}
\Delta^{k-1}_\sigma(t)\leq 2 \int_{T_{k-1}}^{T_k \wedge \tau^\e_T} |A_i(x^\e_{T_{k-1}}) - \mathcal{R}_0^*[H] \mathcal{R}_0[H](q^\e_s)| ds.
\end{equation*}
By perceeding as for $\Delta^{k-1}_b(t)$, we have for $t \in [T^\e_{k-1} \wedge \tau^\e_T, T^\e_k \wedge \tau^\e_T]$ and every k except one,
\begin{equation*}
\mathbb{E}[\Delta_\sigma^k(t) | \mathcal{F}_{T^\e_{k}}] \leq C_1 \bar{C}\e \mathbb{E}_{q^\e_{T^\e_k}}[\sup_{0\leq s \leq \bar{C}\e}|q^\e_s - \bar{X}^{\e,k}_s|].
\end{equation*}

\subsection{Closeness of $\xi^{\e}_t$ and $x^\e_t$}
Let \[h(\e)=\sup_{x\in \mathbb{R}^2}\max \{|b^\e(x)|, |\sigma^{\e}(x)|\},\] then $\lim_{\e \to 0} h(\e) = 0$ as we assumed in Hypothesis \ref{h1}. For $t \in [T^\e_k\wedge \tau^\e_T , T^\e_{k+1}\wedge \tau^\e_T)$, we have
\begin{equation*}
\begin{aligned}
&|\xi^\e_t- x^\e_t| \\
\leq& |B_i(x^\e_{T^\e_{k}})(t-T^\e_k) - \int_{T^\e_k}^{t} \mathcal{L}_0 H (q^\e_s) ds| + |A_i^{1/2}(x^\e_{T^\e_k}) (\beta_t-\beta_{T^\e_k}) - \int_{T^\e_k}^t \mathcal{R}_0 H (q^\e_s) dWs|\\
&+ |\int_{T^\e_k}^t \mathcal{L}^{\e}_0 H (q^\e_s) ds| + |\int_{T^\e_k}^t \mathcal{R}^{\e}_0 H (q^\e_s) dWs|\\
\leq& |B_i(x^\e_{T^\e_{k}})(t-T^\e_k) - \int_{T^\e_k}^{t} \mathcal{L}_0 H (q^\e_s) ds| + |A_i^{1/2}(x^\e_{T^\e_k}) (\beta_t-\beta_{T^\e_k})- \int_{T^\e_k}^{T^\e_{k+1}\wedge \tau} \mathcal{R}_0 H (q^\e_s) dWs| \\
&+\bar{C} \e h(\e) + |\int_{T^\e_k}^t \mathcal{R}^{\e}_0 H (q^\e_s) dWs| \\.
\end{aligned}
\end{equation*} 
If $T^\e_{k} > \tau^\e_T $, then 
\begin{equation*}
\sup_{T^\e_k \wedge \tau^\e_T \leq t \leq T^\e_{k+1} \wedge \tau^\e_T} |\xi^\e_t- x^\e_t|^4=0.
\end{equation*}
So we consider the case when $T^\e_{k} > \tau^\e_T $,  
\begin{equation*}
\sup_{T^\e_k \wedge \tau^\e_T \leq t \leq T^\e_{k+1} \wedge \tau^\e_T} |\xi^\e_t- x^\e_t|^4=\sup_{T^\e_k \leq t \leq T^\e_{k+1} \wedge \tau^\e_T} |\xi^\e_t- x^\e_t|^4,
\end{equation*}
below.
\begin{equation*}
\begin{aligned}
&\sup_{T^\e_k  \leq t \leq T^\e_{k+1} \wedge \tau^\e_T} |\xi^\e_t- x^\e_t|^4 \\
\leq& C (\Delta^k_b(T^\e_{k+1}\wedge \tau^\e_T))^4 + C \sup_{T^\e_k \leq t \leq T^\e_{k+1} \wedge \tau^\e_T} |\int_{T^\e_k}^{t} A_i^{1/2}(x^\e_{T^\e_k} \wedge \tau^\e_T) V^\e_s - \mathcal{R}_0[H](q^\e_s) dW_s |^4\\
&+C\bar{C}^4\e^4 h^4(\e)  + C \sup_{T^\e_k \leq t \leq T^\e_{k+1} \wedge \tau^\e_T}  |\int_{T^\e_k}^t \mathcal{R}^{\e}_0[H](q^\e_s) dWs|^4\\
\leq&  C (\Delta^k_b)^4 + C \sup_{T^\e_k \leq t \leq T^\e_{k+1} \wedge \tau^\e_T} |\int_{T^\e_k}^{t} A_i^{1/2}(x^\e_{T^\e_k}) V^\e_s - \mathcal{R}_0 H (q^\e_s) dW_s |^4\\
&+C \bar{C}^4\e^4 h^4(\e)  + C \sup_{T^\e_k \leq t \leq T^\e_{k+1} \wedge \tau^\e_T}  |\int_{T^\e_k}^t \mathcal{R}^{\e}_0 H (q^\e_s) dWs|^4.
\end{aligned}
\end{equation*}
Therefore, by Kolmogorov's inequality
\begin{equation*}
\begin{aligned}
     &\mathbb{E}[\sup_{T^\e_k \leq t \leq T^\e_{k+1} \wedge \tau} |\xi^\e_t- x^\e_t|^4 | \mathcal{F}_{T^\e_k}]\\
\leq& C \mathbb{E}[ (\Delta_b^k)^4 | \mathcal{F}_{T^\e_k}] + C \mathbb{E}_{{q^\e_{T^\e_k}}}[\int_0^{T^\e_0\wedge \tau^\e_T} |A_i(x^\e_{T^\e_k}) - \mathcal{R}_0^* H \mathcal{R}_0 H (q^\e_s)|ds]^2 + C\bar{C}^4\e^4 h^4(\e) + 2\bar{C}^2 \e^2 h(\e)^2 \\
\leq&  C \mathbb{E}[ (\Delta_b^k)^4 | \mathcal{F}_{T^\e_k}] + C\mathbb{E}[\Delta_{\sigma}^k | \mathcal{F}_{T^\e_k} ]^2 + C\bar{C}^4\e^4 h^4(\e)+ C\bar{C}^2 \e^2 h(\e)^2  \\
\leq&  C^2 (\bar{C} \e)^4 \mathbb{E}[\sup_{0 \leq s \leq \bar{C} \e} |q^\e_s - \bar{X}^\e_s|^4] + C \bar{C}^2\e^2 \mathbb{E}[\sup_{0\leq s \leq \bar{C}\e}|q^\e_s - \bar{X}^\e_s|]^2+  2\bar{C}^4\e^4 h^4(\e)+ 2\bar{C}^2 \e^2 h(\e)^2  \\
\leq& C^4 (\bar{C} \e)^4 \mathbb{E}[\sup_{0 \leq s \leq \bar{C} \e} |q^\e_s - \bar{X}^\e_s|^4] + C \bar{C}^2\e^2 \mathbb{E}[\sup_{0 \leq s \leq \bar{C} \e} |q^\e_s - \bar{X}^\e_s|^4] ^{1/2}+ 2\bar{C}^4\e^4 h^4(\e)+ 2\bar{C}^2 \e^2 h(\e)^2.
\end{aligned}
\end{equation*}
Therefore, thanks to \eqref{eq closeness of trajectory}
\begin{equation*}
\mathbb{E}[\sup_{T^\e_k \leq t \leq T^\e_{k+1} \wedge \tau^\e_T} |\xi^\e_t- x^\e_t|^4 | \mathcal{F}_{T^\e_k}] \leq C ( (\bar{C} \e)^{5} + (\bar{C}\e)^{2.5}) + 2\bar{C}^4\e^4 h^4(\e)+ 2\bar{C}^2 \e^2 h^2(\e).
\end{equation*}
Finally, since there are at most $N:=[T/\underline{C}\e]+1$ $T_k$'s before time time $T$, we get
\begin{equation*}
\begin{aligned}
\mathbb{E}[\sup_{0 \leq t \leq \tau^\e_T} |\xi^\e_t- x^\e_t|^4] 
\leq&\mathbb{E}[\sum_{k=0}^{\infty} \sup_{T^\e_k\wedge \tau^\e_T \leq t \leq T^\e_{k+1}\wedge \tau^\e_T} |\xi^\e_t- x^\e_t|^4] \\
\leq&\mathbb{E}[\sum_{k=0}^{N}\sup_{T^\e_k \leq t \leq T^\e_{k+1}\wedge \tau^\e_T} |\xi^\e_t- x^\e_t|^4] + \mathbb{E}[\sup_{T^\e_N\wedge \tau^\e_T \leq t \leq T^\e_{N+1}\wedge \tau^\e_T} |\xi^\e_t- x^\e_t|^4| \mathcal{F}_{T_N}]]\\
\leq& N (C ( (\bar{C} \e)^{5} + (\bar{C}\e)^{2.5}) + 2\bar{C}^4\e^4 h^4(\e)+ 2\bar{C}^2 \e^2 h^2(\e))\\
\leq&  C\frac{T}{\e} (C ( (\bar{C} \e)^{5} + (\bar{C}\e)^{2.5}) + 2\bar{C}^4\e^4 h^4(\e)+ 2\bar{C}^2 \e^2 h^2(\e)),\\
\end{aligned}
\end{equation*}
So that
\begin{equation*}
\mathbb{E}[\sup_{0 \leq t \leq \tau^\e_T} |\xi^\e_t- x^\e_t|^4]  \leq CT (\bar{C}^5\e^{4} + \bar{C}^{2.5} \e^{1.5} + 2\bar{C}^4\e^3 h^4(\e)+ 2\bar{C}^2 \e h^2(\e)).
\end{equation*}

Therefore, we have proved the following Lemma
\begin{Lemma}\label{weak convergence 1}
Under Hypothesis \ref{h1} and \ref{h2}, for any intitial condition $(x_0,i)$ on edge $I_i$ and $H_1 \leq x_0 \leq H_2$ such that $(H_1,i), (H_2,i) \in I_i$, we have
\begin{equation}\label{eq weak convergence 1}
\sup_{(x_0 ,i ) \in I_i}\mathbb{E}_{x_0}[\sup_{0\leq s \leq \tau^\e_{H_1,H_2} \wedge T } |\xi^\e_t- x^\e_t|^4] \leq  CT (\bar{C}^5\e^{4} + \bar{C}^{2.5} \e^{1.5} + 2\bar{C}^4\e^3 h^4(\e)+ 2\bar{C}^2 \e h^2(\e)).
\end{equation}
\end{Lemma}

\subsection{Closeness of $\xi^\e_t$ and $\tilde{x}^\e_t$}
Let $|\xi^\e_{T^\e_k}-\tilde{x}^\e_{T^\e_k}|= \delta_{k}$. For $t\in [T^\e_{k-1}, T^\e_{k})$, we have
\begin{equation*}
\begin{aligned}
|\xi^\e_t- \tilde{x}^\e_t| \leq& |\xi^\e_{T^\e_{k-1}}-\tilde{x}^\e_{T^\e_{k-1}}| + \int_{T^\e_{k-1}}^{t} |B_i(x^\e_{T^\e_{k-1}})- B_i(\tilde{x}^\e_s)| ds + |\int_{T^\e_{k-1}}^t (A_i(x^\e_{T^\e_{k-1}}) - A_i (\tilde{x}^\e_s)) d\beta^\e_s|\\
\leq& \delta_{k-1} + C \int_{T^\e_{k-1}}^t |x^\e_{T^\e_{k-1}} - \tilde{x}^\e_s| ds + |\int_{T^\e_{k-1}}^t (A_i(x^\e_{T^\e_{k-1}}) - A_i (\tilde{x}^\e_s)) d\beta^\e_s|\\
\leq& \delta_{k-1} + C \int_{T^\e_{k-1}}^t |\xi^\e_{T^\e_{k-1}}-\xi^\e_s|ds + C \int_{T^\e_{k-1}}^t  |\xi^\e_s-\tilde{x}^\e_s| ds + |\int_{T^\e_{k-1}}^t (A_i(\xi^\e_{T^\e_{k-1}}) - A_i (\xi^\e_s)) d\beta^\e_s|\\
&+  |\int_{T^\e_{k-1}}^t (A_i(\xi^\e_{s}) - A_i (\tilde{x}^\e_s)) d\beta^\e_s|.
\end{aligned}
\end{equation*}
Let $\beta^{\e,s}_t= \beta^\e_{t+s}-\beta^\e_s$. For $s\in [T^\e_{k-1}, T^\e_k)$, 
\begin{equation*}
|\xi^\e_{T^\e_{k-1}}-\xi^\e_s| \leq |B_i(x^\e_{T^\e_{k-1}}) (s-T^\e_{k-1})| + |A_i^{\frac{1}{2}}(x^\e_{T^\e_{k-1}})(\beta^\e_s- \beta^\e_{T^\e_{k-1}})|  \leq C|s-T^\e_{k-1}| + C|\beta_{s-T^\e_{k-1}}^{\e,{T^\e_{k-1}}}|.
\end{equation*}
Since $T^\e_k-T^\e_{k-1} \leq \bar{C}\e$, we define $\bar{\xi}^{\e}_s$ to be the extension of $\xi^{\e}_s$ in the interval $ [T^\e_{k-1}, T^\e_{k-1}+\bar{C}\e]$ such that $\bar{\xi^\e_{s}}= \xi^\e_s$, $s\in [T^\e_{k-1}, T^\e_k)$ and $\bar{\xi}^\e_s=0$, $T^\e_k \leq s \leq T^\e_{k-1}+\bar{C}\e$. The process $\bar{\tilde{x}}^\e_s$ is defined similarly. Then for $T^\e_{k-1}< t_0<T^\e_{k-1} + \bar{C}\e$, 
\begin{equation*}
\sup_{T^\e_{k-1} \leq s \leq T^\e_k\wedge t_0}|\xi^\e_s-\tilde{x}^\e_s| = \sup_{T^\e_{k-1} \leq s \leq t_0} |\bar{\xi}^\e_s-\bar{\tilde{x}}^\e_s|,
\end{equation*}
and
\begin{equation*}
\begin{aligned}
     &\sup_{T^\e_{k-1}\leq s \leq t_0 \wedge T^\e_k} |\xi^\e_s - \tilde{x}^\e_s| \\
\leq& \delta_{k-1} + C\int_{T^\e_{k-1}}^{T^\e_{k}\wedge t_0} (|s-T^\e_{k-1}| + |\beta_{s-T^\e_{k-1}}^{\e,{T^\e_{k-1}}}|) ds + C\int_{T^\e_{k-1}}^{T^\e_k \wedge t_0} \sup_{0\leq u \leq s\wedge T^\e_k} |\xi^\e_u- \tilde{x}^\e_u| ds \\
&+|\int_{T^\e_{k-1}}^{T^\e_k \wedge t_0} (A_i(\xi^\e_{T^\e_{k-1}}) - A_i (\xi^\e_s)) d\beta^{\e,T^\e_{k-1}}_s| + |\int_{T^\e_{k-1}}^{T^\e_k \wedge t_0} (A_i(\xi^\e_{s}) - A_i (\tilde{x}^\e_s)) d\beta^{\e,T^\e_{k-1}}_s|.
\end{aligned}
\end{equation*}
So
\begin{equation*}
\begin{aligned}
&\mathbb{E}[\sup_{T^\e_{k-1}\leq s \leq t_0 \wedge T^\e_k} |\xi^\e_s-\tilde{x}^\e_s|^4 | \mathcal{F}_{T^\e_{k-1}}] \\
\leq& c \delta_{k-1}^4 +c \mathbb{E}_{q^\e_{T^\e_{k-1}}}[|\int_0^{T^\e_0} s ds|^4 ] + c \mathbb{E}_{q^\e_{T^\e_{k-1}}}[|\int_0^{T^\e_0}|\beta_s^{\e,{T^\e_{k-1}}}|ds|^4 ] \\ 
&+c \mathbb{E}_{q^\e_{T^\e_{k-1}}}[|\int_0^{T^\e_0 \wedge \tau} \sup_{0\leq u \leq s\wedge{T^\e_0}} |\xi^\e_u-x_u|ds|^4  ] + c \mathbb{E}_{q^\e_{T^\e_{k-1}}}[ |\int_{0}^{T^\e_0 \wedge t_0} (A_i(\xi^\e_{T^\e_{k-1}}) - A_i (\xi^\e_s)) d\beta^\e_s|^4  ] \\
&+ c \mathbb{E}_{q^\e_{T^\e_{k-1}}}[|\int_{0}^{T^\e_0 \wedge t_0} (A_i(\xi^\e_{s}) - A_i (\tilde{x}^\e_s)) d\beta^\e_s|^4].\\
\end{aligned}
\end{equation*}
Let $r= t_0-T^\e_{k-1}$, then $r \leq \bar{C}\e$ and
\begin{equation*}
\begin{aligned}
 &\mathbb{E}[\sup_{T^\e_{k-1}\leq s \leq t_0 \wedge T^\e_k} |\xi^\e_s-\tilde{x}^\e_s|^4 | \mathcal{F}_{T^\e_{k-1}}] =\mathbb{E}[\sup_{T^\e_{k-1}\leq s \leq t_0 } |\bar{\xi}^\e_s-\bar{\tilde{x}}^\e_s|^4 | \mathcal{F}_{T^\e_{k-1}}]\\
\leq&  \d_{k-1}^4 + (\bar{C} \e)^8 + \bar{C}^4 \e^6  +   \mathbb{E}_{q^\e_{T^\e_{k-1}}}[|\int_0^{r} \sup_{0\leq u \leq s} |\bar \xi^\e_u-\bar{\tilde{x}}^\e_u|ds|^4 ]\\
&  + c \mathbb{E}_{q^\e_{T^\e_{k-1}}}[ |\int_{0}^{T^\e_0 \wedge r} (A_i(\bar{\xi}^\e_{T^\e_{k-1}}) - A_i (\bar{\xi}^\e_s)) d\beta^{\e,T^\e_{k-1}}_s|^4  ] + c \mathbb{E}_{q^\e_{T^\e_{k-1}}}[|\int_{0}^{T^\e_0 \wedge r} (A_i(\bar{\xi}^\e_{s}) - A_i (\bar{\tilde{x}}^\e_s)) d\beta^{\e,T^\e_{k-1}}_s|^4]\\
\leq& \d_{k-1}^4 + (\bar{C} \e)^8 + \bar{C}^4 \e^6  +   \mathbb{E}_{q^\e_{T^\e_{k-1}}}[|\int_0^{r} \sup_{0\leq u \leq s} |\bar \xi^\e_u-\bar{\tilde{x}}^\e_u|ds|^4 ]\\
&  + c \mathbb{E}_{q^\e_{T^\e_{k-1}}}[ |\int_{0}^{ r } (A_i(\bar{\xi}^\e_{T^\e_{k-1}}) - A_i (\bar{\xi}^\e_s)) d\beta^{\e,T^\e_{k-1}}_s|^4  ] + c \mathbb{E}_{q^\e_{T^\e_{k-1}}}[|\int_{0}^{ r} (A_i(\bar{\xi}^\e_{s}) - A_i (\bar{\tilde{x}}^\e_s)) d\beta^{\e,T^\e_{k-1}}_s|^4].\\
\end{aligned}
\end{equation*}
Since
\begin{equation*}
\begin{aligned}
\mathbb{E}[|\int_0^r (A_i(\bar{\xi}^\e_{T^\e_{k}}) - A_i (\bar{\xi}^\e_s)) d\beta^{\e}_s|^4] \leq& \bar{C}\e \mathbb{E}[\int_0^r |A_i(\bar{\xi}^\e_{T^\e_{k}}) - A_i (\bar{\xi}^\e_s)|^4 ds]\\
\leq& L^4 \bar{C}\e \mathbb{E}[\int_0^r |\bar{\xi}^\e_{T^\e_{k}} -\bar{\xi}^\e_s|^4 ds]\\
\leq& C \bar{C}\e \int_0^r \mathbb{E}[|\bar{\xi}^\e_{T^\e_{k}} -\bar{\xi}^\e_s|^4 ] ds\\
\leq& C \bar{C}\e \int_0^r \mathbb{E} [s^4 + |\beta^\e_s|^4]ds\\
\leq&C (\bar{C}\e)^4,
\end{aligned}
\end{equation*}
we get
\begin{equation*}
\mathbb{E}[|\int_0^r (A_i(\bar{\xi}^\e_{s}) - A_i (\bar{\tilde{x}}^\e_s)) d\beta^{\e}_s|^4] \leq \bar{C}\e \int_0^r \mathbb{E}[|\bar{\xi}^\e_{s} - \bar{\tilde{x}}^\e_s|^4]ds \leq \bar{C}\e \int_0^r \mathbb{E}[\sup_{0\leq u\leq s}|\bar{\xi}^\e_{u} - \bar{\tilde{x}}^\e_u|^4]ds 
\end{equation*}
and 
\begin{equation*}
\mathbb{E}[\sup_{T^\e_{k-1}\leq s \leq t_0 } |\bar{\xi}^\e_s-\bar{x}_s|^4 | \mathcal{F}_{T^\e_{k-1}}]  =
\mathbb{E}_{q^\e_{T^\e_{k-1}}}[\sup_{0\leq s \leq t } |\bar{\xi}^\e_s-\bar{x}_s|^4 ] ].
\end{equation*}
Therefore 
\begin{equation*}
\begin{aligned}
\mathbb{E}_{q^\e_{T^\e_{k-1}}}[\sup_{0\leq s \leq r } |\bar{\xi}^\e_s-\bar{\tilde{x}}^\e_s|^4 ] ]\leq& \d_{k-1}^4 + (\bar{C} \e)^8 + \bar{C}^4 \e^6  + (\bar{C} \e)^4 + ((\bar{C} \e)^4 + \bar{C} \e)\int_0^r \mathbb{E}_{q^\e_{T^\e_{k-1}}}[\sup_{0\leq u \leq s} |\bar \xi^\e_u-\bar{\tilde{ x}}^\e_u|^4 ].
\end{aligned}
\end{equation*}
If we apply Gronwall's inequality, we have
\begin{equation*}
\mathbb{E}_{q^\e_{T^\e_{k-1}}}[\sup_{0\leq s \leq r } |\bar{\xi}^\e_s-\bar{\tilde{x}}^\e_s|^4 ] ]  \leq
(\d_{k-1}^4 + (\bar{C} \e)^8 + \bar{C}^4 \e^6  + (\bar{C} \e)^4 ) \exp(((\bar{C} \e)^4 + \bar{C} \e)t_0)
\end{equation*}
and thus
\begin{equation*}
\mathbb{E}[\sup_{T^\e_{k-1}\leq s \leq t_0 } |{\xi}^\e_s-\tilde{x}^\e_s|^4 | \mathcal{F}_{T^\e_{k-1}}] ds \leq
(\d_{k-1}^4 + (\bar{C} \e)^8 + \bar{C}^4 \e^6  + (\bar{C} \e)^4 ) \exp((\bar{C} \e)^5 + (\bar{C} \e)^2)
\end{equation*}
All this implies,
\begin{equation*}
\begin{aligned}
&\mathbb{E}[\sup_{T^\e_{k-1}\wedge \tau^\e_T \leq s \leq T^\e_{k}\wedge \tau^\e_T } |\bar{\xi}^\e_s-\bar{\tilde{x}}^\e_s|^4 ]
=  \mathbb{E}[\mathbb{E}[\sup_{T^\e_{k-1}\wedge \tau^\e_T \leq s \leq T^\e_{k}\wedge \tau^\e_T } |\bar{\xi}^\e_s-\bar{\tilde{x}}^\e_s|^4 ] | \mathcal{F}_{T^\e_{k-1}}]\\
\leq& \mathbb{E}[(\d_{k-1}^4 + (\bar{C} \e)^8 + \bar{C}^4 \e^6 + (\bar{C} \e)^6 + (\bar{C} \e)^4 ) \exp((\bar{C} \e)^4 + (\bar{C} \e)^2)]\\
\leq& \mathbb{E}[\sup_{T^\e_{k-2}\wedge \tau^\e_T \leq s \leq T^\e_{k-1}\wedge \tau^\e_T } |\bar{\xi}^\e_s-\bar{x}_s|^4 ]]  \exp((\bar{C} \e)^4 + (\bar{C} \e)^2) \\&+ ( (\bar{C} \e)^8 + \bar{C}^4 \e^6 + (\bar{C} \e)^6 + (\bar{C} \e)^4 ) \exp((\bar{C} \e)^4 + (\bar{C} \e)^2)].\\
\end{aligned}
\end{equation*}
Now, let $a_k= \mathbb{E}[\sup_{T^\e_{k-1}\wedge \tau^\e_T \leq s \leq T^\e_{k}\wedge \tau^\e_T } |\bar{\xi}^\e_s-\bar{\tilde{x}}^\e_s|^4 ]] $. The calculation above implies the following recursive ,
\begin{equation*}
a_k \leq \exp(\bar{C}\e)(a_{k-1} + (\bar{C}\e)^4).
\end{equation*}
By induction, we can easily deduce that
\begin{equation*}
a_n \leq  (\bar{C}\e)^4 \sum_{k=1}^{n}  \exp(k\bar{C}\e) = (\bar{C}\e)^4 \frac{1-\exp((n+1)\bar{C}\e)}{1- \exp(\bar{C}\e)}\leq C \bar{C}^4\e^3.
\end{equation*}
Finally, let $N = [t/\underline{C}\e]+1$,
\begin{equation*}
\begin{aligned}
\mathbb{E}[\sup_{0\leq s \leq \tau^\e_T } |{\xi}^\e_s-\tilde{x}^\e_s|^4] 
\leq& \mathbb{E}[\sup_{k}\sup_{T^\e_{k-1}\wedge \tau^\e_T \leq s \leq T^\e_{k}\wedge \tau^\e_T } |\bar{\xi}^\e_s-\bar{\tilde{x}}^\e_s|^4 ]
\leq \mathbb{E}[\sum_{k=1}^{N}\sup_{T^\e_{k-1}\wedge \tau^\e_T \leq s \leq T^\e_{k}\wedge \tau^\e_T } |\bar{\xi}^\e_s-\bar{\tilde{x}}^\e_s|^4 ]\\
\leq& \sum_{k=1}^{N} \mathbb{E}[\sup_{T^\e_{k-1}\wedge \tau^\e_T \leq s \leq T^\e_{k}\wedge \tau^\e_T } |\bar{\xi}^\e_s-\bar{\tilde{x}}^\e_s|^4 ]
=\sum_{k=1}^{N} a_k \leq C T \bar{C}^4\e^2.
\end{aligned}
\end{equation*}
Therefore, we have proved the following Lemma
\begin{Lemma}\label{weak convergence 2}
Under Hypothesis \ref{h1} and \ref{h2}, for any intitial condition $(x_0,i)$ on edge $I_i$ and $H_1 \leq x_0 \leq H_2$ such that $(H_1,i), (H_2,i) \in I_i$, we have
\begin{equation}\label{eq weak convergence 2}
\sup_{(x_0 ,i ) \in I_i}\mathbb{E}_{x_0}[\sup_{0\leq s \leq \tau^\e_T } |{\xi}^\e_s-\tilde{x}^\e_s|^4] \leq C T \bar{C}^4\e^2.
\end{equation}
\end{Lemma}

\subsection{Proof of Lemma \ref{convergence in edge weaker version}}
By Ito's formula, for any $\e>0$,
\begin{equation*}
\mathbb{E}_{x_0}[e^{-\la \tau^\e_T}f(x_{\tau^\e_T}) - \int_0^{+\infty} e^{-\la s}(-\la I + \mathcal{L}_i)[f](x_{s}) \mathcal{X}_{\tau^\e_T>s} ds ] = f(x_0).
\end{equation*}
Since $\tilde{x}^\e_t$ and $x_t$ have the same distribution,
\begin{equation*}
\mathbb{E}_{x_0}[e^{-\la\tau^\e_T}f(\tilde{x}^\e_{\tau^\e_T}) - \int_0^{+\infty} e^{-\la s}(-\la I + \mathcal{L}_i)[f](\tilde{x}^\e_{s})\mathcal{X}_{\tau^\e_T>s} ds ] = f(x_0).
\end{equation*}
Moreover,
\begin{equation*}
\begin{aligned}
&\mathbb{E}_{x_0}[e^{-\la \tau^\e_T}(f(x^\e_{\tau^\e_T}) - f(\tilde{x}^\e_{\tau^\e_T}))]^2 \\
\leq& \mathbb{E}_{x_0} [|f(x^\e_{\tau^\e_T}) - f(\tilde{x}^\e_{\tau^\e_T}))|^2] \leq |f|_{Lip}^2\mathbb{E}_{x_0} [ |x^\e_{\tau^\e_T} - \tilde{x}^\e_{\tau^\e_T}|^2]\\
\leq& |f|_{Lip}^2(\mathbb{E}_{x_0} [ |\xi_{\tau^\e_T} - \tilde{x}^\e_{\tau^\e_T}|^4]^{\frac{1}{2}}+\mathbb{E}_{x_0} [ |x^\e_{\tau^\e_T} - \xi^\e_{\tau^\e_T}|^4]^{\frac{1}{2}})\\
\leq& |f|_{Lip}^2(\mathbb{E}_{x_0} [ \sup_{0 \leq t \leq \tau^\e_T}|\xi_{t} - \tilde{x}^\e_{t}|^4]^{\frac{1}{2}}+\mathbb{E}_{x_0} [  \sup_{0 \leq t \leq\tau^\e_T}|x^\e_{t} - \xi^\e_{t}|^4]^{\frac{1}{2}}).\\
\end{aligned}
\end{equation*}
Apply Lemma \ref{weak convergence 1} and \ref{weak convergence 2}, $\mathbb{E}_{x_0}[e^{-\la \tau^\e_T}(f(x^\e_{\tau^\e_T}) - f(\tilde{x}^\e_{\tau^\e_T}))]$ converges to 0 uniformly with respect to $x_0$. And
\begin{equation*}
\begin{aligned}
&\mathbb{E}_{x_0}[|\int_0^{\tau^\e_T} e^{-\la s}((-\la I + \mathcal{L}_i)[f](x^\e_{s}) - (-\la I + \mathcal{L}_i)[f](\tilde{x}^\e_{s}))ds |]^2\\
\leq& \mathbb{E}_{x_0}[\int_0^{\tau^\e_T} e^{-\la s}C_f |x^\e_s- \tilde{x}^\e_s| ds]^2 \leq C_f^2 T^2 \mathbb{E}_{x_0}[\sup_{0 \leq t \leq \tau^\e_T}|x^\e_{t} - \tilde{x}^\e_{t}|^2])\\
\leq& C_f^2 T^2 (\mathbb{E}_{x_0} [ \sup_{0 \leq t \leq \tau^\e_T}|\xi_{t} - \tilde{x}^\e_{t}|^4]^{\frac{1}{2}}+\mathbb{E}_{x_0} [  \sup_{0 \leq t \leq\tau^\e_T}|x^\e_{t} - \xi^\e_{t}|^4]^{\frac{1}{2}}).\\
\end{aligned}
\end{equation*}
Similarly, $\mathbb{E}_{x_0}[|\int_0^{\tau^\e_T} e^{-\la s}((-\la I + \mathcal{L}_i)[f](x^\e_{s}) - (-\la I + \mathcal{L}_i)[f](\tilde{x}^\e_{s}))ds |]$ converges to 0 uniformly with respect to $x_0$. So
\begin{equation*}
\begin{aligned}
&\lim_{\e \to 0}\mathbb{E}_{x_0}[e^{-\la \tau^\e_T}f(x_{\tau^\e_T}) - \int_0^{\tau^\e_T} e^{-\la s}(-\la I + \mathcal{L}_i)[f](x_{s}) ds ] \\
=&\lim_{\e \to 0} \mathbb{E}_{x_0}[e^{-\la\tau^\e_T}f(\tilde{x}^\e_{\tau^\e_T}) - \int_0^{\tau^\e_T} e^{-\la s}(-\la I + \mathcal{L}_i)[f](\tilde{x}^\e_{s}) ds ] =f(x_0).
\end{aligned}
\end{equation*}

\subsection{Proof of Lemma \ref{convergence in edge}}
We start with the following Lemma
\begin{Lemma}
Under Hypothesis \ref{h1} and \ref{h2}, and let the stopping time $\tau^\e_{H_1,H_2}$ be defined as before, where $H_1$, $H_2$ are in the interior of $H(I_i)$, then
\begin{equation*}
\sup_{\e>0} \sup_{ z \in D_i([H_1,H_2])}\mathbb{E}_z[\tau^\e_{H_1,H_2}] < \infty.
\end{equation*}
\end{Lemma}
\begin{proof}
Let $f$ satisfies $\mathcal{L}_i [f]=-1$ and the boundary conditions $f(H_1)=f(H_2)=0$. Then apply Lemma \ref{convergence in edge} with $f$ and have
\begin{equation*}
\mathbb{E}[\tau^\e_T]=f(x_0)- \mathbb{E}[f(x^\e_{\tau^\e_T})]
\end{equation*}
Since the second order differential operator $\mathcal{L}_i$ is uniformly elliptic on $[H_1, H_2]$, we can apply the maximum principle, and we get
\begin{equation*}
|f|_{\infty} \leq C, 
\end{equation*}
where $C$ is a constant depending on the coefficients in $\mathcal{L}_i$ and $[H_1,H_2]$. In particular, it is independent of $T$. Therefore, by Fatou's Lemma
\begin{equation*}
\mathbb{E}_z[\tau^\e_{H_1,H_2}] = \mathbb{E}_z[ \lim_{T \to \infty}\tau^\e_T] \leq   \liminf_{T \to \infty} \mathbb{E}[\tau^\e_T] \leq 2C < \infty.
\end{equation*}
Since this inequality holds for all $z \in D_i([H_1,H_2])$ and $\e >0$.
\end{proof}

\begin{Lemma}\label{uniform weak convergence 1}
Under Hypothesis \ref{h1} and \ref{h2}, for any intitial condition $(x_0,i)$ on the edge $I_i$ and $H_1 \leq x_0 \leq H_2$ such that $(H_1,i), (H_2,i) $ are inside the interior of $I_i$, we have
\begin{equation*}
\lim_{\e \to 0}\sup_{(x_0 ,i ) \in I_i}\mathbb{E}_{x_0}[\sup_{0\leq s \leq \tau^\e_{H_1,H_2} } |\xi^\e_t- x^\e_t|^2]=0.
\end{equation*}
\end{Lemma}
\begin{proof}
Unlike the its previous counterpart, the number of intervals here is random. Let $N$ be the largest $k$ such that $T^\e_k \leq \tau^\e_{H_1,H_2}$. For simplicity, we write $\tau^\e_{H_1,H_2}$ as $\tau^\e_T$, eliminate the edge coordinate $i$ and let $T_{-1}=0$ in this proof. Thanks to \eqref{eq weak convergence 1} and apply Fatou's Lemma and strong Markov property,
\begin{equation*}
\begin{aligned}
\mathbb{E}_{x_0}[\sup_{0 \leq t \leq \tau^\e_T} |\xi^\e_t- x^\e_t|^2] 
\leq&\mathbb{E}_{x_0}[\sum_{k=-1}^{\infty} \sup_{T_k\wedge \tau^\e_T \leq t \leq T_{k+1}\wedge \tau^\e_T} |\xi^\e_t- x^\e_t|^2]\\ \leq& \sum_{k=-1}^{\infty}\mathbb{E}_{x_0}[\sup_{T_k \leq t \leq T_{k+1}\wedge \tau^\e_T}|\xi^\e_t- x^\e_t|^2 \mathcal{X}_{k \leq N}]\\
\leq&\sum_{k=-1}^{\infty}\mathbb{E}_{x_0}[\mathbb{E}_{x_0}[\sup_{T_k \leq t \leq T_{k+1}\wedge \tau^\e_T}|\xi^\e_t- x^\e_t|^2 \mathcal{X}_{T_k \leq \tau^\e_T}|\mathcal{F}_{T_k}]]\\
=&\sum_{k=-1}^{\infty}\mathbb{E}_{x_0}[\mathbb{E}_{x_0}[\sup_{T_k \leq t \leq T_{k+1}\wedge \tau^\e_T}|\xi^\e_t- x^\e_t|^2 |\mathcal{F}_{T_k}]\mathcal{X}_{T_k \leq \tau^\e_T}]\\
\leq& \sum_{k=-1}^{\infty}\mathbb{E}_{x_0}[\mathbb{E}_{H(q^\e_{T_k})}[\sup_{0 \leq t \leq T_0 \wedge \tau^\e_T}|\xi^\e_t- x^\e_t|^2]\mathcal{X}_{T_k \leq \tau^\e_T}]\\
\leq& \sum_{k=-1}^{\infty}\mathbb{E}_{x_0}[\sup_{(x_0,i) \in I_i}\mathbb{E}_{x_0}[\sup_{0 \leq t \leq T_0 \wedge \tau^\e_T}|\xi^\e_t- x^\e_t|^2]\mathcal{X}_{T_k \leq \tau^\e_T}]\\
=& \sum_{k=-1}^{\infty}\sup_{(x_0,i) \in I_i}\mathbb{E}_{x_0}[\sup_{T_k \leq t \leq T_{k+1}\wedge \tau^\e_T}|\xi^\e_t- x^\e_t|^2]\mathbb{E}_{x_0}[\mathcal{X}_{T_k \leq \tau^\e_T}]\\
\leq& (C ( (\bar{C} \e)^{5} + (\bar{C}\e)^{2.5}) + 2\bar{C}^4\e^4 h^4(\e)+ 2\bar{C}^2 \e^2 h^2(\e)) \sum_{k=0}^{\infty} \mathbb{E}[\mathcal{X}_{k \leq N} ].\\
\end{aligned}
\end{equation*}
Moreover,
\begin{equation*}
\sum_{k=0}^{\infty} \mathbb{E}[\mathcal{X}_{k \leq N}] \leq 1+\mathbb{E}[N]= 1+ \mathbb{E}[\frac{\tau^\e_{H_1,H_2}}{\underline{C}\e}] \leq \frac{C}{\e}
\end{equation*}
and thus prove for all $T>0$,
\begin{equation*}
\mathbb{E}_{x_0}[\sup_{0 \leq t \leq \tau^\e_T} |\xi^\e_t- x^\e_t|^2]  \leq C(C ( (\bar{C} \e)^{4} + (\bar{C}\e)^{3/2}) + 2\bar{C}^4\e^3 h^4(\e)+ 2\bar{C}^2 \e h^2(\e)).
\end{equation*}
Finally, we take the $\sup$ over $T$ and prove the result.
\end{proof}

\begin{Lemma}\label{uniform weak convergence 2}
Under Hypothesis \ref{h1} and \ref{h2}, for any intitial condition $(x_0,i)$ on edge $I_i$ and $H_1 \leq x_0 \leq H_2$ such that $(H_1,i), (H_2,i)$ inside the interior of $I_i$, we have
\begin{equation*}
\lim_{\e \to 0}\sup_{(x_0 ,i ) \in I_i}\mathbb{E}_{x_0}[\sup_{0\leq s \leq \tau^\e_{H_1,H_2}} |{\xi}^\e_s-\tilde{x}^\e_s|^2] =0.
\end{equation*}
\end{Lemma}
\begin{proof}
Similarly, we have
\begin{equation*}
\begin{aligned}
\mathbb{E}_{x_0}[\sup_{0\leq s \leq \tau^\e_T } |{\xi}^\e_s-\tilde{x}^\e_s|^2] 
\leq& \mathbb{E}_{x_0}[\sup_{k}\sup_{T_{k-1}\wedge \tau^\e_T \leq s \leq T_{k}\wedge \tau^\e_T } |\bar{\xi}^\e_s-\bar{\tilde{x}}^\e_s|^2 ]\\
\leq & \mathbb{E}_{x_0}[\sum_{k=0}^{\infty}\sup_{T_{k-1}\wedge \tau^\e_T \leq s \leq T_{k}\wedge \tau^\e_T } |\bar{\xi}^\e_s-\bar{\tilde{x}}^\e_s|^2 \mathcal{X}_{k \leq N} ]\\
\leq& \sum_{k=0}^{\infty} \mathbb{E}_{x_0}[\sup_{T_{k-1}\wedge \tau^\e_T \leq s \leq T_{k}\wedge \tau^\e_T } |\bar{\xi}^\e_s-\bar{\tilde{x}}^\e_s|^2 \mathcal{X}_{k \leq N} ]\\
\leq& \sum_{k=0}^{\infty} \sup_{(x_0,i)\in I_i}\mathbb{E}_{x_0}[\sup_{T_{k-1}\wedge \tau^\e_T \leq s \leq T_{k}\wedge \tau^\e_T } |\bar{\xi}^\e_s-\bar{\tilde{x}}^\e_s|^4] \mathbb{E}[ \mathcal{X}_{k \leq N} ].\\
\end{aligned}
\end{equation*}
Now by the estimation \eqref{eq weak convergence 2}, and process as the previous Lemma, the result is obvious.
\end{proof}

Now if we apply exactly the same proof as in Lemma \ref{convergence in edge weaker version}, Lemma \ref{convergence in edge} is proved.

\section{Properties near the Saddle Point}\label{gluing condition}
Since the gluing condition is given by local property, we can assume without loss of generality in this section that
\begin{Hypothesis}\label{h3}
The coefficients $g(x)$, $b(x)$, $b^\e(x)$, $\si(x)$ and $\si^\e(x)$, $x \in \mathbb{R}^2$ in equation (\ref{qsde}) are uniformly bounded. 
\end{Hypothesis}

\subsection{Averaging the Measure}
We first remark that: the probability for $q^\e_t$, starting from an initial point $q \in D_k(\pm \d')$ to reach the level set $C_{ki}(\d)$ is approximately the same. The following Lemma is a consequence of Krylov and Safonov's theory (see \cite{KS}).
\begin{Lemma}\label{averaging measure 1}
There exists a $\d^{\ref{averaging measure 1}}>0$ such that for every $0< \d'< \d <\d^{\ref{averaging measure 1}}$ and and any $\kappa>0$, the following estimate holds for sufficiently small $\e$,
\begin{equation*}
\sup_{q_1, q_2 \in \bar{D}_k(\pm \d')} |\mathbb{E}_{q_1}[f(q^\e_{\tau^\e_k(\pm \d)})]-\mathbb{E}_{q_2}[f(q^\e_{\tau^\e_k(\pm \d)})]| < \kappa.
\end{equation*}
\end{Lemma}
If the function $f$ in the Lemma is $\mathcal{X}_{C_{kj}(\d)}(I_j \sim O_k)$, we have $\mathbb{E}_{q_i}[f(q^\e_{\tau^\e_k(\pm \d)})]= \mathbb{P}_{q_i}[q^\e_{\tau^\e_k(\pm \d)} \in C_{kj}(\d)]$. Now we can average the probability $\mathbb{P}_{q}[q^\e_{\sigma_1} \in C_{kj(\d)}]$ by a measure $\mu$ on the level set $C_k(\pm \d')$.
\begin{Lemma}\label{averaging measure 2}
Let $\mu$ be a measure on the level set $C_k(\pm \d')$. For any $\kappa>0$ and sufficiently small $\e$,
\begin{equation}\label{eq28}
\sup_{x \in C_k(\pm \d')}|\mathbb{P}_x[q^\e_{\tau^\e_k(\pm \d) \in C_{ki}(\d)}]- \int_{C_k(\pm \d')} \frac{\mu(dx) \mathbb{P}_x[q^\e_{\tau^{\e}_k(\pm \d)} \in C_{ki}(\d)]}{\mu(C_k(\pm \d'))}|<\kappa.
\end{equation}
\end{Lemma}
The measure $\mu$ will be determined by the long term behavior of $q^\e_t$ as to be shown in the next subsection.

\subsection{Representation of the Invariant Measure}
Let 
\[
C_\d= \bigcup_{k \in S} \bigcup_{I_i \sim O_k}C_{ki}(\pm \d)\ \ \ \ \ \ C_{\d'}= \bigcup_{k \in S} \bigcup_{I_i \sim O_k} C_{ki}(\pm \d').
\]
The sequence $\{q^\e_{\tau^\e_k}\}_{k=0}^{\infty}$ is a Markov chain and for all $k>1$, $q^\e_{\tau^\e_k} \in C_{\d'}$. Moreover if $q^\e_0\in C_{\d}$, $\{ q^\e_{\sigma^\e_k} \}_{k=0}^{\infty}$ is a Markov chain on $C_{\d}$. Then every invariant measure $\mu^\e$ of the process $(q^\e_t, \mathbb{P}^\e_q)$ can be represented in the form
\begin{equation}、\label{eq26}
\mu^\e(A)= \int_{C_\d} \nu^\e(dz) \mathbb{E}_z^\e[\int_0^{\sigma^\e_1} \mathbbm{1}_A(q^\e_t) dt]= \int_{C_{\d'}} \nu'^{\e} (dz) \mathbb{E}^\e_z[\int_0^{\tau^\e_1} \mathbbm{1}_A (q^\e_t) dt],
\end{equation}
where $\nu^\e$ and $\nu'^\e$ are measures on $C_\d$ and $C_{\d'}$ such that
\begin{equation}\label{eq27}
\begin{aligned}
\nu^\e(B)= \int_{C_{\d'}} \nu'^\e (dz) \mathbb{P}^\e_z[q^\e_{\sigma^\e_0} \in B],\ \ \ \nu'^\e(C)= \int_{C_\d} \nu^\e(dz) \mathbb{P}^\e_z[q^\e_{\tau^\e_1} \in C].
\end{aligned}
\end{equation}
See \cite{Khas}. Notice that $\nu^\e$ and $\nu'^{\e}$ are the invariant measures for the Markov chains $\{q^\e_{\si^\e_k}\}_{k=0}^{\infty}$ and $\{q^\e_{\tau^\e_k}\}_{k=0}^{\infty}$ on $C_\d$ and $C_{\d'}$ respectively. Now we pick the measure $\mu$ in equation (\ref{eq28}) to be $\nu'^\e$ and we get
\begin{equation*}
 \int_{C_k(\pm \d')} \frac{\mu(dz) \mathbb{P}_z[q^\e_{\tau^{\e}_k(\pm \d)} \in C_{ki}(\d)]}{\mu(C_k(\pm \d'))}=\frac{\nu^\e(C_{ki}(\d)}{\nu'^\e(C_k(\pm \d'))}.
\end{equation*}
Hence, we apply equation \eqref{eq27} with $C= C_k(\pm \d')$ and since $\mathbb{P}_z^\e[q^\e_{\tau_1} \in C]=1$, for all $z \in C_k(\pm \d‘) $, we have $\nu'^\e(C_{k}(\pm \d'))= \nu^\e(C_{k}(\pm \d'))$. Therefore
\begin{equation*}
|\mathbb{P}_z[q^\e_{\tau^\e(\pm \d) \in C_{ki}(\d)}]-\frac{\nu^\e(C_{ki}(\d)}{\nu^\e(C_k(\pm \d))}|<\kappa.
\end{equation*}
Finally, apply equation \eqref{eq26} with a bounded measurable function $G$ and get
\begin{equation}\label{eq32}
\int_{\mathbb{R}^2} G(z) \mu^\e(dz)= \int_{C_\d} \nu^\e(dz) \mathbb{E}_z^\e[\int_0^{\sigma_1} G(q^\e_t) dt].
\end{equation}

\subsection{Proof of Lemma \ref{lemma gluing condition}: A Special Case}
In this section we first assume that the following Hypothesis holds,
\begin{Hypothesis}\label{h4}
Let function $a(x)$ be the function defined by \eqref{invariant measure}. We assume that $a$ satisfies the following relation for all $\e>0$.
\begin{equation*}
(\mathcal{L}_0)^*[a^{-1}](z)=0,\ \ \ (\mathcal{L}_0^\e)^*[a^{-1}](z)=0\ \ \ \forall z \in \mathbb{R}^2.
\end{equation*}
\end{Hypothesis}
Then $a^{-1}(z)dz$ is the invariant measure for the the system $\{q^\e_t\}$ as $\e \downarrow 0$ and $\mathcal{L}_i$ is a generalized differential operator. Apply Lemma \ref{l15} again, the left hand side of equation (\ref{eq32}) becomes
\begin{equation*}
\int_{\mathbb{R}^2} G(z) \mu^\e(dz)= \int_{\mathbb{R}^2} G(z) \frac{dz}{a(z)}= \int \int_{C(H)} G(z) \frac{dl}{|\nabla H(z)|a(z)} dH,
\end{equation*}
Let the window function $G(z)$ depend only on the Hamiltonian, 
\[
G(z)= f \circ H(z),
\]
where $f \in \textbf{C}^{\infty}(\mathbb{R})$ with support in $(H_{k1}+\d,H_{k2}-\d)$ and $H_{ki}=H(O_{ki})$, $i=1,2$, $O_{ki}(i=1,2)$ are the two end points of an edge $I_k$. Then
\begin{equation}\label{eq51}
\int_{\mathbb{R}^2} G(z) \mu^\e(dz)= \int_{H_1+\d}^{H_2-\d} f(H) v'_k(H) dH.
\end{equation}
Finally, we apply Lemma \ref{convergence in edge} to the right hand side of equation \eqref{eq32},
\begin{equation*}
\lim_{\e \to 0} \mathbb{E}_z^\e[\int_0^{\sigma^\e_1} G(q^\e_t) dt]=\lim_{\e \to 0} \mathbb{E}_x^\e[\int_0^{\sigma^\e_1} f(x^\e_t) dt]=\lim_{\e \to 0} \mathbb{E}_x^\e[\int_0^{\tau^\e_1} f(x^\e_t) dt]= u(x),
\end{equation*}
where $x= H(z)$ and $u(x)$ is the solution to
\begin{equation*}
\begin{cases}
\mathcal{L}_i[u]=f\\
u(H_1+\d')=u(H_2-\d')=0.
\end{cases}
\end{equation*}
It can be easily seen that $u(x)$ can be solved by the following formula,
\begin{equation*}
\begin{aligned}
u(x)=&\frac{u_k(H_{k2}-\d')-u_k(x)}{u_k(H_{k2}-\d')-u_k(H_{k1}+\d')}\int_{H_{k1}+\d'}^{x}(u_k(x)-u_k(H_{k1}+\d')) f(h) dv_k(h)\\
&+\frac{u_k(x)-u_k(H_1+\d')}{u_k(H_{k2}-\d')-u_k(H_{k1}+\d')}\int_{x}^{H_{k2}-\d'} (u_k(H_{k2}-\d')-u_k(x)) f(h) dv_k(h).
\end{aligned}
\end{equation*}
Notice that $f$ has support in $(H_{k1}+\d,H_{k2}-\d)$, so if $x \in C_{k1}(\d)$, then
\begin{equation*}
u(H_{k1}+\d)=\frac{u_k(H_1+\d)-u_k(H_1+\d')}{u_k(H_{k2}-\d')-u_k(H_{k1}+\d')}\int_{H_{k1}+\d}^{H_{k2}-\d'} (u_k(H_{k2}-\d')-u_k(H_{k1}+\d)) f(h) dv_k(h),
\end{equation*}
and if $x \in C_{k2}(-\d)$, then
\begin{equation*}
u(H_{k2}-\d)=\frac{u_k(H_{k2}-\d')-u_k(H_{k2}-\d)}{u_k(H_{k2}-\d')-u_k(H_{k1}+\d')}\int_{H_{k1}+\d'}^{H_{k2}-\d}(u_k(H_{k2}-\d)-u_k(H_{k1}+\d')) f(h) dv_k(h).
\end{equation*}
Now we combine equation \eqref{eq32} and equation \eqref{eq51},
\begin{equation*}
\int_{H_1+\d}^{H_2-\d} f(H) v'_k(H) dH = \lim_{\e \to 0} \nu^\e (C_{k1}(\d)) (u(H_{k1}+\d)+\mathcal{O}(\e) )+ \nu^\e(C_{k2}(\d)) (u(H_{k2}-\d)+\mathcal{O}(\e)).
\end{equation*}
which is an equation holds for all $f$, therefore we have:
\begin{equation*}
\lim_{\e \to 0}\nu^\e(C_{k1}(\d)) = \frac{1}{u_i(H_1+\d)-u_i(H_2+\d')},
\end{equation*}
\begin{equation*}
\lim_{\e \to 0}\nu^\e(C_{k2}(\d)) = \frac{1}{u_i(H_2-\d')-u_i(H_2-\d)}.
\end{equation*}
Finally, letting $\d \to \d'$ and $\d'\to 0$ gives us Lemma \ref{lemma gluing condition} under Hypothesis \ref{h4}.

\subsection{Proof of Lemma \ref{lemma gluing condition}: The General Case}

Now we consider the case without Hypothesis \ref{h4}.
Instead of considering the process $q^\e_t$ whose invariant measure is unknown when $\e \downarrow 0$, we consider the system by compensating a smooth vector filed $\hat{b}$ and $\hat{b}^\e$ satisfying 
\begin{equation}\label{eq30}
(\mathcal{L}_0 + \hat{b}\cdot \nabla)^*[a^{-1}]=0,\ \ \ (\mathcal{L}^\e_0 + \hat{b}^\e\cdot \nabla)^*[a^{-1}]=0
\end{equation}
and $\mathcal{L}_0+ \hat{b} \nabla$, $(\mathcal{L}_0^\e+ \hat{b}^\e \nabla)^*$ are the corresponding formal adjoint operator. We first claim that both such $\hat{b}$ and $\hat{b}^\e$ exists under Hypothesis \ref{h1}, \ref{h2}, and \ref{h3}. Indeed 
\begin{equation*}
\begin{aligned}
&(\mathcal{L}_0 + \hat{b}\cdot \nabla)^*[a^{-1}]=\mathcal{L}_0[a^{-1}] - \sum_{i=1}^2 \frac{\partial}{\partial x_i}(\hat{b}_i a^{-1}).
\end{aligned}
\end{equation*}
Let $\hat{b}_2=0$ and $\hat{b}_1$ satisfy the first-order partial differential equation
\begin{equation*}
\frac{\partial}{\partial x_1}(\hat{b}_1 a^{-1})=\mathcal{L}_0[a^{-1}](x).
\end{equation*}
Then
\begin{equation}\label{eq31}
\hat{b}_1(x_1,x_2)=a(x_1,x_2) \int_0^{x_1}\mathcal{L}_0[a^{-1}](u_1,x_2) du_1.
\end{equation}
By Hypothesis \ref{h3}, $\hat{b}$ is well-defined, bounded and globaly Lipschitz continuous. Similarly, such $\hat{b}^\e$ also exists. Therefore the stochastic process $\hat{q}^\e_t$ is well-defined for every fixed $\e>0$:
\begin{equation*}
\begin{cases}
d\hat{q}^\e_t=\frac{1}{\e} g(\hat{q}^\e_t) dt +b(\hat{q}^\e_t) dt + \sigma(\hat{q}^\e_t) dW_t + b^\e(\hat{q}^\e_t) dt + \sigma^\e(\hat{q}^\e_t) dW_t + \hat{b}(\hat{q}^\e_t) dt + \hat{b}^\e(\hat{q}^\e_t) dt\\
\hat{q}^\e_0=q,
\end{cases}
\end{equation*}
and $\hat{q}^\e_t $ satisfies Hypothesis \ref{h4}. Therefore, the gluing condition holds for $\hat{q}^\e_t$.

For the process $q^\e_t$, the drift term $\hat{b}$ and $\hat{b}^\e$ need to be killed from $\hat{q}^\e_t$. I apply Girsanov Theorem up to any fixed time $T$. 

The measure $\mathbb{P}^\e_q$ induced by the process $q^\e_t$ on $\mathbf{C}([0,T];\mathbb{R}^2)$ is absolutely continuous to the measure $\hat{\mathbb{P}}^\e_q$ induced by $\hat{q}^\e_t$ on $\mathbf{C}([0,T]; \mathbb{R}^2)$ with the density
\begin{equation*}
\frac{d \mathbb{P}^\e_q}{d \hat{\mathbb{P}}^\e_q}|_{\mathcal{F}_t}= \exp \{ \int_0^t M^\e(\hat{q}^\e_s) dW_s - \frac{1}{2} \int_0^t |M^\e(\hat{q}^\e_s)|^2 ds\},\ \ \ 0\leq t \leq T
\end{equation*}
where
\begin{equation*}
M^\e(x)= -(\sigma +\sigma^\e)^{-1}(x) (\hat{b}+\hat{b}^\e)(x).
\end{equation*}
Compare the two measures of a set $A \in \mathcal{F}_T$, where the expectation, without specification, is taken with respect to the measure $\hat{\mathbb{P}}_q^\e$.
\begin{equation*}
\begin{aligned}
|\mathbb{P}^\e_q[A]- \hat{\mathbb{P}}^\e_q[A]|
=& |\int_A [\exp \{ \int_0^T M^\e(\hat{q}^\e_s) dW_s - \frac{1}{2} \int_0^T |M^\e(\hat{q}^\e_s)|^2 ds\}-1] d\hat{\mathbb{P}^\e_q}[\omega]|\\
\leq& (\hat{\mathbb{P}}^\e_q[A])^{1/2} {\mathbb{E}}^\e_q[ [ \int_0^T M^\e(\hat{q}^\e_s) dW_s - \frac{1}{2} \int_0^T |M^\e(\hat{q}^\e_s)|^2 ds-1]^2]^{1/2}.
\end{aligned}
\end{equation*}
For simplicity, we denote 
\begin{equation*}
Z^\e_t= \int_0^t M^\e(\hat{q}^\e_s) dW_s - \frac{1}{2} \int_0^t |M^\e(\hat{q}^\e_s)|^2 ds.
\end{equation*}
Apply Ito's formula to $[\exp(Z_t)-1]^2$:
\begin{equation*}
\begin{aligned}
&\mathbb{E}[\exp(Z^\e_t)-1]^2\\
=& \mathbb{E}[0 + \int_0^t 2[\exp(Z^\e_s)-1]\exp(Z^\e_s)dZ^\e_s \\
&+ \frac{1}{2} \int_0^t [2\exp(Z^\e_s)^2+2\exp(Z^\e_s)(\exp(Z^\e_s)-1)] d\langle Z^\e_s,Z^\e_s \rangle]\\
=& \mathbb{E}[\int_0^t \exp(2Z^\e_s)|M(\hat{q}^\e_s)|^2ds ].
\end{aligned}
\end{equation*}
By Hypothesis \ref{h3} and \ref{h4}, 
\begin{equation*}
\sup_{\e>0, x \in \mathbb{R}^2}|M^\e(x)|^2 = M
\end{equation*} 
and
\begin{equation*}
\begin{aligned}
&\mathbb{E}[[\exp(Z^\e_t)-1]^2]=\mathbb{E}[\int_0^t \exp(2Z^\e_s)|M^\e(\hat{q}^\e_s)|^2ds]\leq  M \mathbb{E}[\int_0^t \exp(2Z^\e_s)ds]\\
=& M \mathbb{E}[\int_0^t \exp(2\int_0^s M^\e(\hat{q}^\e_u)dW_u-\int_0^s |M^\e(\hat{q}^\e_u)|^2 du)]ds\\
=& M \mathbb{E}[\int_0^t \exp(\int_0^s 2M^\e(\hat{q}^\e_u)dW_u-\frac{1}{2}\int_0^s |2M^\e(\hat{q}^\e_u)|^2 du)\exp(\int_0^s |M^\e(\hat{q}^\e_u)|^2du)ds].\\
\end{aligned}
\end{equation*}
Then apply Girsanov theorem again to the last equation
\begin{equation*}
\begin{aligned}
&\mathbb{E}[\int_0^t \exp(\int_0^s 2M^\e(\hat{q}^\e_u)dW_u-\frac{1}{2}\int_0^s |2M^\e(\hat{q}^\e_u)|^2 du)\exp(\int_0^s |M^\e(\hat{q}^\e_u)|^2du)ds]\\
=& \int_0^t \tilde{\mathbb{E}}[\exp \int_0^s |M^\e(\hat{q}^\e_u)|^2du]ds\\
\leq&  C (e^{M T}-1).
\end{aligned}
\end{equation*}
This implies
\begin{equation*}
\hat{\mathbb{E}}^\e_q[ [ \int_0^T M^\e(\hat{q}^\e_s) dW_s - \frac{1}{2} \int_0^T |M^\e(\hat{q}^\e_s)|^2 ds-1]^2] \leq M_0 T.
\end{equation*}
Then let $T=\d$ and for $q\in \bar{D}_k(\pm \d'_g)$ as in Lemma \ref{lemma gluing condition},
\begin{equation*}
\begin{aligned}
&|\mathbb{P}^\e_q[q^\e_{\tau^\e_k(\pm \d)} \in C_{ki}(\d)]- p_{ki}|\\
\leq& |\mathbb{P}^\e_q[q^\e_{\tau^\e_k(\pm \d)} \in C_{ki}(\d)]- \hat{\mathbb{P}}^\e_q[q^\e_{\tau^\e_k(\pm \d)} \in C_{ki}(\d)]| + |\hat{\mathbb{P}}^\e_q[q^\e_{\tau^\e_k(\pm \d)} \in C_{ki}(\d)]- p_{ki}|\\
\leq&  |\mathbb{P}^\e_q[q^\e_{\hat{\tau}^\e_k(\pm \d)} \in C_{ki}(\d); \hat{\tau}^\e_k(\pm \d) < \d] - \hat{\mathbb{P}}^\e_q[\hat{q}^\e_{\tau^\e_k(\pm \d)} \in C_{ki}(\d) ; \hat{\tau}^\e_k(\pm \d) < \d]\\
& + 2 \mathbb{P}^\e_q[\hat{\tau}^\e_k(\pm \d) > \d]+ \kappa\\
\leq& \hat{\mathbb{P}}^\e_q[\hat{q}^\e_{\tau^\e_k(\pm \d)} \in C_{ki}(\d) ]^{1/2} \cdot M \d^{1/2} + 2 \mathbb{P}^\e_q[\hat{\tau}^\e_k(\pm \d) > \d]+ \kappa,
\end{aligned}
\end{equation*}
and thanks to Lemma \ref{not sticky lemma}, we have 
\begin{equation*}
\mathbb{E}^\e_q[\hat{\tau}^\e_k(\pm \d) ] \leq M \d^2 \ln(\d).
\end{equation*}
Therefore, for sufficiently small $\d$,
\begin{equation*}
|\mathbb{P}^\e_q[q^\e_{\tau^\e(\pm \d)} \in C_{ki}(\d)]- p_{ki}| \leq M \d^{1/2} +  M \d \ln(\d) + \kappa < 3\kappa.
\end{equation*}
Therefore, we proved the general case.

\end{document}